\documentclass[11pt]{amsart}
\usepackage{geometry}                
\usepackage{mathrsfs}
\usepackage{graphicx}
\usepackage{amssymb}
\usepackage{epstopdf}
\usepackage{amsmath}
\usepackage{color}
\usepackage{hyperref}
\usepackage{pdfsync}
\usepackage{float}
\usepackage{verbatim}
\usepackage{tikz}
\usetikzlibrary{arrows,decorations.pathmorphing,backgrounds,positioning,fit,petri}
\usetikzlibrary{arrows,patterns,matrix}
\usetikzlibrary{decorations.markings}

\usepackage[all]{xy}
\xyoption{matrix}
\xyoption{arrow}
\usepackage[table]{colortbl}
 
\tikzset{help lines/.style={step=#1cm,very thin, color=gray},
help lines/.default=.5} 
\tikzset{thick grid/.style={step=#1cm,thick, color=gray},
thick grid/.default=1} 

\usepackage{tcolorbox}  
\tcbset{colframe=black,colback=white, width=(\linewidth-2mm), before=,after=\hfill,} 
\usepackage{graphicx}
\usepackage{tikz}
\usetikzlibrary{mindmap,shapes,snakes,matrix}

\tikzstyle{ann}=[fill=white, inner sep=1pt, font=\footnotesize{#1}]
\tikzstyle{annfar}=[inner sep=2pt, font=\footnotesize{#1}]
\tikzstyle{annfarer}=[inner sep=3pt, font=\footnotesize{#1}]
\tikzstyle{annrot}=[fill=white, text=blue!75!black, inner sep=1pt, font=\footnotesize{#1}]
\tikzstyle{wall}=[thick]
\tikzstyle{nullwall}=[thick, dotted]

\newtheorem{thm}{Theorem}[subsection]
\newtheorem{thm*}{Theorem}
\newtheorem{lem}[thm]{Lemma}
\newtheorem{cor}[thm]{Corollary}
\newtheorem{prop}[thm]{Proposition}

\theoremstyle{definition}
\newtheorem{defn}[thm]{Definition}

\newtheorem{exm}[thm]{Example}

\newtheorem{rem}[thm]{Remark}
\newtheorem{rmk}[thm]{Remark}

\numberwithin{equation}{section}

\DeclareMathOperator{\ind}{ind}
\DeclareMathOperator{\Coker}{Coker}
\DeclareMathOperator{\Hom}{Hom}%
\DeclareMathOperator{\Ext}{Ext}%
\DeclareMathOperator{\Tor}{Tor}%
\DeclareMathOperator{\End}{End}%
\DeclareMathOperator{\add}{add} 
\DeclareMathOperator{\pd}{projdim}%
\DeclareMathOperator{\id}{injdim}%
\DeclareMathOperator{\gldim}{gldim}%
\DeclareMathOperator{\domdim}{domdim}%
\DeclareMathOperator{\repdim}{repdim}%
\DeclareMathOperator{\findim}{findim}%
\DeclareMathOperator{\modd}{mod}%
\DeclareMathOperator{\Gdim}{Gdim}%
\DeclareMathOperator{\Gen}{Gen}%
\DeclareMathOperator{\Cogen}{Cogen}%
\DeclareMathOperator{\Sub}{Sub}%

\newcommand{\commentout}[1]{}

\newcommand{\cC}{\ensuremath{{\mathcal{C}}}}

\newcommand{\cF}{\ensuremath{{\mathcal{F}}}}

\newcommand{\cP}{\ensuremath{{\mathcal{P}}}}
\newcommand{\cQ}{\ensuremath{{\mathcal{Q}}}}

\newcommand{\cT}{\ensuremath{{\mathcal{T}}}}

\newcommand{\cX}{\ensuremath{{\mathcal{X}}}}
\newcommand{\cY}{\ensuremath{{\mathcal{Y}}}}

\newcommand{\cx}{\ensuremath{{\mathcal{X}}}}
 
 \newcommand{\cc}{\ensuremath{{\mathcal{C}}}}

\newcommand{\Top}{\operatorname{top}}
\newcommand{\rad}{\operatorname{rad}}
\newcommand{\soc}{\operatorname{soc}}
\newcommand{\Ker}{\operatorname{Ker}}
\newcommand{\cok}{\operatorname{Coker}}
\newcommand{\Ima}{\operatorname{Im}}

\newcommand{\bk}{\mathbf{k}}

\newcounter{margincounter}
\title[Dominant dimension and tilting modules]{Dominant dimension and tilting modules}
\author{Van C.~Nguyen}
\address{Department of Mathematics, Hood College, Frederick, MD 21701, USA}
\email{nguyen@hood.edu}

\author{Idun Reiten}
\address{Institutt for matematiske fag, Norges Teknisk-Naturvitenskapelige Universitet, N-7491 Trondheim, Norway}
\email{idun.reiten@ntnu.no}

\author{Gordana Todorov}
\address{Department of Mathematics, Northeastern University, Boston, MA 02115, USA}
\email{g.todorov@northeastern.edu}
\author{Shijie Zhu}
\address{Department of Mathematics, Northeastern University, Boston, MA 02115, USA}
\email{zhu.shi@husky.neu.edu}

\date{\today}                      


\begin{document}


\begin{abstract} 
We study which algebras have tilting modules that are both generated and cogenerated by projective-injective modules. 
Crawley-Boevey and Sauter have shown that Auslander algebras have such tilting modules; and for algebras of global dimension $2$, Auslander algebras are classified by the existence of such tilting modules.

In this paper, we show that the existence of such a tilting module is equivalent to the algebra having dominant dimension at least $2$, independent of its global dimension. In general such a tilting module is not necessarily cotilting. Here, we show that the algebras which have a tilting-cotilting module generated-cogenerated by projective-injective modules are precisely $1$-Auslander-Gorenstein algebras.

When considering such a tilting module, without the assumption that it is cotilting, we study the global dimension of its endomorphism algebra, and discuss a connection with the Finitistic Dimension Conjecture. Furthermore, as special cases, we show that triangular matrix algebras obtained from Auslander algebras and certain injective modules, have such a tilting module. We also give a description of which Nakayama algebras have such a tilting module.

\end{abstract}

\maketitle
\tableofcontents

 
 \section*{Introduction}\label{sec: Introduction}

   Let $\Lambda$ be an artin algebra and let $\modd \Lambda$ be the category of finitely generated left $\Lambda$-modules.  Throughout the paper, $\gldim \Lambda$ denotes the global dimension of $\Lambda$, $\domdim \Lambda$ denotes its dominant dimension (c.f. Definition \ref{dom.dim.defn}), and $\Gdim \Lambda$ denotes its Gorenstein dimension  (c.f. Remark \ref{gdim defn}). For any $\Lambda$-module $M$, $\pd M$ denotes its projective dimension and $\id M$ denotes its injective dimension. 
   
   Let $\widetilde{Q}$
be the direct sum of representatives of the isomorphism classes of all indecomposable projective-injective $\Lambda$-modules. Let $\mathcal C_\Lambda: = (\Gen \widetilde{Q}) \cap (\Cogen \widetilde{Q})$ be the full subcategory of $\modd \Lambda$ consisting of all modules generated and cogenerated by $\widetilde{Q}$. 
When $\gldim \Lambda=2$, Crawley-Boevey and Sauter showed in \cite[Lemma 1.1]{CBS} that the algebra $\Lambda$ is an Auslander algebra if and only if there exists a tilting $\Lambda$-module $T_\cC$ in $\mathcal C_{\Lambda}$.
In fact, $T_\cC$ is the direct sum of representatives of the isomorphism classes of indecomposable modules in $\mathcal C_\Lambda$. Furthermore $T_\cC$ is the unique tilting module in $\mathcal C_{\Lambda}$ and it is also a cotilting module.

There is another characterization of Auslander algebras as algebras $\Lambda$ such that $\gldim \Lambda \leq 2$ and $\domdim \Lambda \geq 2$. From the above result in \cite{CBS}, it follows that in global dimension 2, the existence of a tilting module in $\mathcal C_{\Lambda}$ is equivalent to $\domdim \Lambda \geq 2$. 
In this paper, we show that the existence of such a tilting module
is equivalent to $\domdim \Lambda \geq 2$
without any condition on the global dimension of $\Lambda$, and we give a precise description of such a tilting module (see Corollary \ref{cotilting cor}, Corollary \ref{T_C compute} and Remark \ref{C_C compute}):


\begin{thm*} Let $\Lambda$ be an artin algebra.  Let $\widetilde{Q}$ be the projective-injective $\Lambda$-module as  above. 
 \begin{enumerate}
\item The following statements are equivalent:
 \begin{enumerate}
  \item[(a)] $\domdim \Lambda\geq 2$,
  \item[(b)] $\mathcal C_\Lambda$ contains a tilting $\Lambda$-module $T_\cC$, 
  \item[(c)] $\mathcal C_\Lambda$ contains a cotilting $\Lambda$-module $C_\cC$. 
 \end{enumerate}
 
\item If a tilting module $T_\cC$ exists, then $T_\cC \simeq \widetilde{Q} \oplus \left(\bigoplus_i \Omega^{-1} P_i\right)$, where $\Omega^{-1} P_i$ is the cosyzygy of $P_i$ and the direct sum is taken over representatives of the isomorphism classes of all indecomposable projective non-injective $\Lambda$-modules $P_i$.
 
\item If a cotilting module $C_\cC$ exists, then $C_\cC \simeq \widetilde{Q} \oplus \left(\bigoplus_i \Omega I_i\right)$, where $\Omega I_i$ is the syzygy of $I_i$ and the direct sum is taken over representatives of the isomorphism classes of all indecomposable injective non-projective $\Lambda$-modules $I_i$.
\end{enumerate} 
\end{thm*}

%
%

Dominant dimensions of algebras under derived equivalences induced by tilting modules were studied by Chen and Xi; in particular they looked at a special class of the so-called \emph{canonical tilting modules} \cite[p.385]{CX} (or canonical $k$-tilting modules to specify the projective dimension being $k$, c.f.~Remark \ref{canonical k tilting}). Recently, the same tilting modules, also called \emph{$k$-shifted modules}, are studied by Pressland and Sauter in \cite{PS}. They show that the existence of a $k$-shifted module is equivalent to the dominant dimension of the algebra being at least $k$. 
We remark that our tilting module $T_\cC$ in $\mathcal C_{\Lambda}$ is a canonical $1$-tilting module. However when $\domdim\Lambda \leq 1$, a canonical $1$-tilting module never belongs to $\cC_\Lambda$. 

In this paper, we concentrate on the existence and properties of the classical tilting module $T_\cC$ in the subcategory $\cC_\Lambda$; in addition to its description, we also consider classes of algebras $\Lambda$ which have such tilting modules in $\cC_\Lambda$.   

Theorem $1$ is proved and discussed in detail in Sections~\ref{subsec:tilting} and \ref{subsec:cotilting}. 
As a generalization of \cite[Lemma 1.1]{CBS}, we describe Auslander algebras as algebras $\Lambda$ with finite global dimension such that there exists a tilting-cotilting module in $\cC_\Lambda$ (see Corollary \ref{tilting-cotilting cor}). More generally, we characterize a larger class, of $1$-Auslander-Gorenstein algebras (c.f.~Definition~\ref{mAusGor}) as: 

\begin{thm*}
Let $\Lambda$ be an artin algebra. Then the subcategory $\mathcal C_\Lambda$ contains a tilting-cotilting module if and only if $\Lambda$ is a $1$-Auslander-Gorenstein algebra.
\end{thm*}

The (sub)structures of classes of such algebras with their homological properties are described in the following diagram (see Definition \ref{mAusGor} and Remark \ref{dtr selfinj} for some definitions): \\

{\Small
\begin{tikzpicture}
\tikzstyle{iellipse}=[draw=black,shape=ellipse,very thick];  
\tikzstyle{ibox}=[draw=black,shape=rectangle,very thick];
\node[ibox, align=center, below] (n1) at (2,0)  {Existence of tilting-cotilting module in $\cC_\Lambda$}; 
\node[ibox, align=center, below] (n2) at (2,-1)  {\\ \textbf{$1$-Auslander-Gorenstein algebras} \\ $\id\,_\Lambda\Lambda \leq 2 \leq \domdim \Lambda$, \\ $ \gldim \Lambda \in \{0,2, \infty\}$ }; 
\node[iellipse, align=center, below] (n3) at (-1.2,-2.8)  {\textbf{$DTr$-selfinjective algebras} \\ $\Gdim \Lambda = 2 = \domdim \Lambda$,\\ $ \gldim \Lambda \in \{2, \infty\}$ };
\node[iellipse, align=center, below] (n4) at (5.2,-2.8)  {\textbf{Selfinjective algebras} \\ $\Gdim \Lambda=0$, $\domdim \Lambda = \infty$,\\ $\gldim \Lambda \in \{0, \infty\}$};
\node[ibox, align=center, below] (n5) at (-3.4,-5)  {\textbf{Auslander algebras} \\ $\Gdim \Lambda = 2 = \domdim \Lambda$,\\ $ \gldim \Lambda = 2$ \\ (Non-semisimple)};
\node[ibox, align=center, below] (n6) at (1,-5)  {\textbf{Non-Auslander} \\ \textbf{$DTr$-selfinjective algebras} \\  $\Gdim \Lambda = 2 = \domdim \Lambda$,\\ $ \gldim \Lambda = \infty$ };

\draw[-, very thick, <->=stealth] (n1) -- (n2);
\draw[-, very thick, >=stealth] (n2) -- (n3);
\draw[-, very thick, >=stealth] (n2) -- (n4);
\draw[-, very thick, >=stealth] (n3) -- (n5);
\draw[-, very thick, >=stealth] (n3) -- (n6);
\end{tikzpicture} 
} 
\vskip 7pt

In Section \ref{DomDim}, we gather further properties of algebras with dominant dimension at least $2$. From the  results in \cite{Mo,M,T}, it follows that such algebras are isomorphic to $\End_\Lambda(X)^{op}$ for some algebra $\Lambda$ and a $\Lambda$-module $X$ which is a generator and a cogenerator; we recall what these algebras $\Lambda$ and modules $X$ should look like and also give a precise description of the tilting module $T_\cC$ in terms of $\Lambda$ and $X$ in Proposition~\ref{XT}. In Section \ref{endomorphism algebra}, given an artin algebra $\Lambda$ with $\gldim\Lambda=d$ and a tilting module $T_\cC\in\cC_\Lambda$ (if it exists), we study the endomorphism algebra $B_\cC:=\End_\Lambda(T_\cC)^{op}$. We show that $d-1\leq\gldim B_\cC\leq d$, and $\gldim B_\cC=d-1$ if and only if $\pd (\tau T_\cC) <d$, (see Corollary \ref{gldim B} and Theorem \ref{drop}). Applying this together with the description of algebras of dominant dimension at least $2$ in Section \ref{DomDim}, we obtain a result about the Finitistic Dimension Conjecture for a certain class of artin algebras of representation dimension at most $4$ in Corollary \ref{findim}. 


In Section \ref{Auslander}, we construct classes of algebras closely related to Auslander algebras which have tilting modules in the subcategory $\mathcal C_{\Lambda}$ of $\modd \Lambda$. More precisely we have:

\begin{thm*} Let $A$ be an Auslander algebra. Let $E$ be an injective $A$-module such that $\End_A(E)$ is a semisimple algebra and $\Hom_A(E,Q)=0$, for all projective-injective $A$-modules $Q$. Then $A[E]$, the triangular matrix algebra of $A$ and the $A$-$\End_A(E)^{op}$-bimodule $E$, has a tilting module in the subcategory $\mathcal C_{A[E]}$.
\end{thm*}

In Section \ref{Nakayama}, 
we use a numerical condition to give a characterization of Nakayama algebras $\Lambda$ which have a tilting module in $\mathcal C_{\Lambda}$. This class of algebras has been classified by Fuller in \cite[Lemma 4.3]{F}; we give a combinatorial approach using Auslander-Reiten theory: 


\begin{thm*} Let $\Lambda$ be a Nakayama algebra with $n$ simple modules. Let $c$ be an admissible sequence of a given Kupisch series. Let the set $\cP_c$ label all indecomposable projective non-injective $\Lambda$-modules, the set $\cQ_c$ label all indecomposable projective-injective $\Lambda$-modules, and $c_j$ be the length of the indecomposable module $P_j$. 

Then there exists a tilting module in $\mathcal C_{\Lambda}$ if and only if $\cP_c \subseteq\{j-c_{j}\in\mathbb Z_n \mid j\in\cQ_c\}$.
\end{thm*} 
The description of such a tilting module is given in Section~\ref{tilting Nakayama}. 

\vskip 5pt
{\bf Acknowledgement:}  We would like to thank Rene Marczinzik, Matthew Pressland, and Julia Sauter for helpful conversations and remarks, especially Matthew for pointing out a mistake in the original proof of Theorem \ref{drop}. The third author would like to thank NTNU for their hospitality during her several visits while working on this project. This work was done when the first author was a Zelevinsky Research Instructor at Northeastern University, she thanks the Mathematics Department for their support. 


\section{Projective-injectives and the subcategory $\mathcal C_{\Lambda}$}

Let $\Lambda$ be an artin algebra and let $\modd\Lambda$ be the category of finitely generated left $\Lambda$-modules.

\begin{defn}
A  $\Lambda$-module $X$ is called a \textbf{generator} if for any $\Lambda$-module $M$, there is an epimorphism $X^m\rightarrow M$, for some $m$. A $\Lambda$-module $X$ is called a \textbf{cogenerator} if for any $\Lambda$-module $M$, there is a monomorphism $M\rightarrow X^m$, for some $m$. 
\end{defn}

We denote by $\Gen (X)$, respectively $\Cogen (X)$, the full subcategories of $\modd\Lambda$ consisting of modules generated by $X$, respectively cogenerated by $X$. Notice that $X$ is a {\bf generator-cogenerator} if and only if each indecomposable projective $\Lambda$-module and indecomposable injective $\Lambda$-module is isomorphic to a direct summand of $X$. 

\begin{defn}\label{tilde Q}
Let $\widetilde{Q}:=\bigoplus_{i=1}^t Q_i$, where the $Q_i$ are representatives of the isomorphism classes of all indecomposable projective-injective $\Lambda$-modules. Let $\mathcal C_{\Lambda}:= (\Gen \widetilde{Q}) \cap (\Cogen \widetilde{Q})$ be the full subcategory of $\modd\Lambda$ consisting of all modules generated and cogenerated by $\widetilde{Q}$. 
\end{defn}

We are going to investigate when there exists a tilting module in $\mathcal C_{\Lambda}$.

\subsection{General properties of the subcategory $\mathcal C_{\Lambda}$}

We now describe some basic homological properties of the modules in the subcategory $\mathcal C_{\Lambda}$ for artin algebras $\Lambda$. 

\begin{prop}\label{k-1}
Let $\Lambda$ be an artin algebra with $\gldim \Lambda = d$. Let $\mathcal C_{\Lambda} = (\Gen \widetilde{Q})\cap (\Cogen \widetilde{Q})$, where $\widetilde{Q}$ is the above projective-injective $\Lambda$-module. Let $X$ be any module in $\mathcal C_{\Lambda}$. Then $\pd X\leq d-1$ and $\id X\leq d-1$.
\end{prop} 
\begin{proof} Since $X$ is in $\mathcal C_{\Lambda}$, there exist short exact sequences:
$$0\rightarrow N \rightarrow Q_0\rightarrow X\rightarrow 0 \qquad \text{ and } \qquad
0\rightarrow X \rightarrow Q'_0\rightarrow L\rightarrow 0,$$
with $Q_0$ and $Q_0'$ projective-injective $\Lambda$-modules. Then there are induced long exact sequences:
$$\dots\rightarrow \Ext^{d}_{\Lambda}(\ ,N) \rightarrow \Ext^{d}_{\Lambda}(\ ,Q_0)\rightarrow \Ext^{d}_{\Lambda}(\ ,X)\rightarrow  \Ext^{d+1}_{\Lambda}(\ ,N) \rightarrow \cdots, \text{ and}$$
$$\dots\rightarrow \Ext^{d}_{\Lambda}(L,\ ) \rightarrow \Ext^{d}_{\Lambda}(Q'_0,\ )\rightarrow \Ext^{d}_{\Lambda}(X,\ )\rightarrow  \Ext^{d+1}_{\Lambda}(L,\ ) \rightarrow \cdots,$$
which show that $\Ext^{d}_{\Lambda}(\ ,X)=0$ and $\Ext^{d}_{\Lambda}(X,\ )=0$, since $Q_0$ is injective, $Q_0'$ is projective and $\gldim \Lambda = d$. Also $\Ext^{j}_{\Lambda}(\ ,X)=0$ and $\Ext^{j}_{\Lambda}(X,\ )=0$, for all $j\geq d$. Hence, $\pd X \leq d-1$ and $\id X \leq d-1$.
\end{proof}

As an Auslander algebra $A$ has $\gldim A \leq 2$, we obtain the following consequence:

\begin{cor} \label{C_A} Let $A$ be an Auslander algebra. Let $X$ be in $\mathcal C_A$. Then $\pd X\leq1$ and $\id X \leq 1$. 
\end{cor}

\begin{prop} \label{PinC} Let $\Lambda$ be an artin algebra. Then:
\begin{enumerate}
\item If $P$ is projective and $P$ is in $\mathcal C_{\Lambda}$, then $P$ is projective-injective.\\
If $I$ is injective and $I$ is in $\mathcal C_{\Lambda}$, then $I$ is projective-injective.
\item Let $X$ be in $\mathcal C_{\Lambda}$.\ Then the projective cover $P(X)$ of $X$ is injective, and the injective envelope $I(X)$ of $X$ is projective. Hence, $P(X)$ and $I(X)$ are in $\mathcal C_{\Lambda}$.
\end{enumerate}
\end{prop}
\begin{proof} (1) Let $P$ be projective in $\mathcal C_{\Lambda}$. Then $P$ is a quotient of a projective-injective module $Q$. Since $P$ is projective, it is a summand of $Q$, and therefore it is injective as well. Dually the injective module $I \in \mathcal C_{\Lambda}$ is also projective. 

(2) Since $X$ is in $\mathcal C_{\Lambda}$, there is a projective-injective module $Q_0$ which maps onto $X$. Thus, the projective cover $P(X)$ is a direct summand of $Q_0$ and so it is injective. Similarly, $I(X)$ is projective by a dual argument.
\end{proof}

\begin{lem}\label{Ext^1.}
Let $\Lambda$ be an artin algebra. Let $X$ be in $\mathcal C_{\Lambda}$. Let $Y$ be a $\Lambda$-module with $\pd Y=1$. Then $\Ext^1_{\Lambda}(Y,X)=0$. 
\end{lem}
\begin{proof} Let $0\to K\to P_0 \to X\to 0$ be an exact sequence with $P_0$ the projective cover of $X$. Consider the induced exact sequence 
$$\dots\to \Ext^1_{\Lambda}(Y,P_0) \to \Ext^1_{\Lambda}(Y,X)\to \Ext^2_{\Lambda}(Y,K)\dots\to.$$ 
Here  $\Ext^1_{\Lambda}(Y,P_0)=0$ since $P_0$ is injective, and $ \Ext^2_{\Lambda}(Y,K)=0$ since $\pd Y = 1$. Therefore, we have $\Ext^1_{\Lambda}(Y,X)=0$ as claimed.
\end{proof}

\subsection{Tilting modules in $\mathcal C_{\Lambda}$} 

Usually, there will be only partial tilting modules in $\mathcal C_{\Lambda}$, and in general there could be no tilting module in $\mathcal C_{\Lambda}$. In this subsection, we show some of the properties of a tilting module in $\mathcal C_{\Lambda}$, if it exists. We recall here the definition of tilting and cotilting modules, since both will be studied extensively in this paper:

\begin{defn}\label{tilt}
Let $\Lambda$ be an artin algebra. A basic $\Lambda$-module $T$ is called \textbf{partial tilting} if it satisfies conditions (1) and (2). It is called \textbf{tilting module} if it satisfies (1), (2), and (3).
\begin{enumerate}
\item $\pd_\Lambda T\leq1$
\item $\Ext_\Lambda^1(T,T)=0$
\item There is an exact sequence $0\rightarrow \Lambda\rightarrow T_0\rightarrow T_1\rightarrow 0$, where $T_0, T_1\in \add T$.
\end{enumerate}
\indent A $\Lambda$-module $C$ is called {\bf cotilting} if it satisfies conditions $(1^o)$, $(2^o)$ and $(3^o)$.

$(1^o)$ $\id_\Lambda C\leq1$

$(2^o)$ $\Ext_\Lambda^1(C,C)=0$

$(3^o)$ There is an exact sequence $0\rightarrow C_0\rightarrow C_1 \rightarrow D\Lambda\rightarrow 0$, where $C_0, C_1\in \add C$.
\end{defn}

\begin{rem} [\cite{ASS}, Corollary VI. 4.4]
Let $n$ be the number of non-isomorphic simple  $\Lambda$-modules. Let $T$ be a partial tilting module. Then the condition (3) is equivalent to: 

$(3')$ The number of non-isomorphic indecomposable summands of $T$ is $n$.
\end{rem}

\begin{rem}\label{canonical k tilting}To avoid confusion, we clarify the use of terminology ``tilting module" here. 
\begin{itemize}
 \item In our definition, ``tilting'' means the classical tilting module as in Definition~\ref{tilt}, with $\pd_\Lambda T\leq1$. In particular, we denote a classical tilting module by $T_\cC$ if it lies in $\cC_\Lambda$. We will prove later that $T_\cC$ is unique, if it exists.

\item In the literature, some authors use the terminology ``tilting modules" for \emph{generalized tilting modules} (e.g. Happel \cite{H}): $(1)$ $\pd T<\infty$, $(2)$ $\Ext^i(T,T)=0$, for all $i>0$, \newline $(3)$ There is an exact sequence $0\rightarrow \Lambda\rightarrow T_0\rightarrow T_1\rightarrow\cdots\rightarrow T_m\rightarrow 0$ for some $m>0$ and $T_i\in\add T$ for all $0\leq i\leq m$.  A generalized tilting module $T$ with $\pd T=k$ is also called \emph{$k$-tilting module} in \cite[Definition 2.3]{PS}.

\item For an algebra $\Lambda$ with dominant dimension at least $k$, Chen and Xi defined in \cite{CX} the \emph{canonical $k$-tilting module} as follows: 
Let $\widetilde{Q}$ be the direct sum of representatives of the isomorphism classes of all indecomposable projective-injective $\Lambda$-modules, and
$$
0\rightarrow\Lambda\stackrel{d_0}\rightarrow I_0\stackrel{d_1}\rightarrow I_1\stackrel{d_2}\rightarrow I_2\rightarrow\cdots
$$
be a minimal injective resolution of $\Lambda$. Then the module $T_{(k)}:=\widetilde{Q}\oplus \Ima d_k$ is a basic $k$-tilting module and it is called the \emph{canonical $k$-tilting module}. 
A \emph{canonical $k$-cotilting module} is defined dually.

%
 \end{itemize}
\end{rem}

\begin{lem} \label{n} Let $\Lambda$ be an artin algebra. Let $n$ be the number of non-isomorphic simple  $\Lambda$-modules. Let $\{X_i\}_{i\in I}$ be a set of indecomposable modules such that $X_i\ncong X_j$ for all $i\neq j$ and $\Ext^1_{\Lambda}(X_i,X_j)=0$ for all $i, j$.  Assume that  $\pd X_i=1$ for all $i\in I$. Then the set $I$ is finite and has at most $n$ elements.
\end{lem}
\begin{proof} Let $X_1,\dots, X_s$ be any $s$ modules in this set. Then $\bigoplus _{i=1}^s X_i$ is a partial tilting module. Every partial tilting module can be completed to a tilting module (see \cite{B}). A tilting module  has $n$ non-isomorphic indecomposable summands. Therefore, $s \leq n$. So there are at most $n$ modules in the set $\{X_i\}_{i\in I}$.
\end{proof}

\begin{prop}\label{rigid} Let $\Lambda$ be an artin algebra. Let $\widetilde{Q}$ be the projective-injective module defined above and 
let $\mathcal C_{\Lambda} = (\Gen \widetilde{Q})\cap (\Cogen \widetilde{Q})$. Let $\{X_i\}_{i\in I}$ be the set of representatives of the indecomposable modules in $\mathcal C_{\Lambda}$ such that $\pd X_i = 1$. Then:
\begin{enumerate}
\item The set $\{X_i\}_{i\in I}$ is finite, that is, $I=\{1,2,\dots , s \}$ for some $s < \infty$.
\item Let $X =\bigoplus _{i=1}^s X_i$. Then $\widetilde{Q} \oplus X$ is a partial tilting module.
\item If there is a tilting module $T_\cC$ in $\mathcal C_{\Lambda}$, then $T_\cC=\widetilde{Q} \oplus X$.
\item If there is a tilting module $T_\cC$ in $\mathcal C_{\Lambda}$, then $T_\cC$ is unique.
\end{enumerate}
\end{prop} 
\begin{proof} (1) It follows from Lemma \ref{Ext^1.} that $\Ext^1_{\Lambda}(X_i,X_j)=0$, for all $i\neq j$.
Since $\pd X_i = 1$, it follows from Lemma \ref{n} that there are at most $n$ modules $X_i$, where $n$ is the number of non-isomorphic simple $\Lambda$-modules. \\
(2) Follows from the definition of partial tilting module.\\
(3) This follows since all other modules in $\mathcal C_{\Lambda}$ have projective dimension $\geq 2$.\\
(4) It follows from (3) that $T_\cC=\widetilde{Q}\oplus X$, hence it is unique.
\end{proof}

The following proposition is about the $\add T_\cC$-approximations of projective modules.

\begin{prop} \label{TapproxP} Let $\Lambda$ be an artin algebra and $P$ be a projective $\Lambda$-module. Suppose there exists a tilting module $T_\cC$ in $\mathcal C_{\Lambda}$. Let $f_P: P\to T_P$ be a minimal left $\add T_\cC$-approximation of $P$. Then $T_P$ is projective-injective.
\begin{proof} Let $f_P: P\to T_P$ be a minimal $\add T_\cC$-approximation of $P$. Then $T_P=Q_P\oplus M_P$, where $Q_P$ is projective-injective and $\pd M_P =1$ and $f_P=(s,\rho):P\to Q_P\oplus M_P$. Let $\sigma: Q_P'\to M_P$ be the projective cover of $M_P$; here $Q_P'$ is projective-injective by Proposition \ref{PinC}(2). Then $\rho$ factors through $\sigma$, i.e. $\rho = \sigma a$ for some $a:P\to Q_P'$. It is easy to check that $(s,a):P\to Q_P\oplus Q_P'$ is an $\add T_\cC$-approximation. Hence $T_P$ is a direct summand of $Q_P\oplus Q_P'$ and therefore it is projective-injective.
\end{proof}
\end{prop}

\section{Dominant dimension and tilting modules (or cotilting modules)}

In this section we show that the existence of a tilting module (or a cotilting module) in the subcategory $\mathcal C_{\Lambda}$ is equivalent to the dominant dimension of $\Lambda$ being at least 2. 

\subsection{Numerical condition} We now state a numerical condition which will be necessary and sufficient for the existence of a tilting module in $\mathcal C_{\Lambda}$.




Let $\mathcal Q:=\add  \widetilde Q$ be the subcategory of $\mathcal C_\Lambda$ consisting of projective-injective modules where $\widetilde Q=\oplus_{i=1}^{t} Q_i$ as in Definition $\ref{tilde Q}$. Let $\cx:=\add X$ be the subcategory of $\mathcal C_\Lambda$ consisting of modules with projective dimension $1$, where $X=\oplus_{i=1}^s X_i$ as in Proposition $\ref{rigid}$. We denote by $n_\cQ$ the number of non-isomorphic indecomposable modules in $\cQ$ and by $n_{\cx}$ the number of non-isomorphic indecomposable modules in $\cx$. Hence, $n_\cQ=t$ and $n_{\cx}=s$.

\begin{rem} Let $n$ be the number of non-isomorphic simple $\Lambda$-modules.  Since by Proposition \ref{rigid}(2), $\widetilde{Q}\oplus X$ is a partial tilting module, it follows that $n_{\mathcal Q}+n_{\cx}\leq n$.
\end{rem}

Combining this  remark and Proposition \ref{rigid} we obtain the following important numerical condition for the existence of a tilting module in $\cc_{\Lambda}$.
\begin{cor}
Let $\Lambda$ be an artin algebra with $n$ non-isomorphic simple modules. Let $\cQ$ and $\cx$ be the above subcategories. Then there is a tilting module in $\cc_{\Lambda}$ if and only if $n_\mathcal Q+n_{\cx}=n$.
\end{cor}
\begin{proof} If there is a tilting module $T_\cC$ in $\cc_{\Lambda}$  then it has $n_{\cQ}$ summands from $\cQ$ and $n_{\cx}$ summands from $\cx$ by Proposition \ref{rigid}(3).
\end{proof}

\subsection{Maps from $\mathcal X$ to projective non-injective modules}

In this part we define a mapping $\mathbf{\Omega}: \ind\cx \rightarrow \ind \mathcal P$, where $\mathcal P$ is the subcategory of projective non-injective $\Lambda$-modules. This mapping will be a bijection exactly when there is a tilting module in $\mathcal C_{\Lambda}$, which will be shown in Corollary \ref{main cor}. This will be used in a very essential way in the proof of the main Theorem \ref{tilting}. 
We need some preparation:

\begin{lem}\cite[II, Lemma 4.3]{A2}\label{minimal}
Let $0\rightarrow A\stackrel{g}\rightarrow B\stackrel{f}\rightarrow C\rightarrow 0$ be a non-split exact sequence in an additive category $\mathcal C$. Then:
\begin{enumerate}
\item If $\End_{\mathcal C}(A)$ is local, then $f:B\rightarrow C$ is right minimal in $\mathcal C$.
\item If $\End_{\mathcal C}(C)$ is local, then $g: A\rightarrow B$ is left minimal in $\mathcal C$.
\end{enumerate}
\end{lem}

\begin{lem} \label{min ind}
Let $0\rightarrow Y\stackrel{g}\rightarrow Q\stackrel{f}\rightarrow X\rightarrow 0$ be a non-split exact sequence. 
\begin{enumerate}
\item Suppose $Y$ is indecomposable, $g$ is left minimal and $Q$ is projective. Then $X$ is indecomposable and $f$ is right minimal.
\item Suppose $X$ is indecomposable, $f$ is right minimal and $Q$ is injective. Then $Y$ is indecomposable and $g$ left minimal.
\end{enumerate}
\end{lem}

\begin{proof}
$(1)$ By Lemma $\ref{minimal}$,  $Y$ being indecomposable implies that $f$ is right minimal. Hence  $f$ is a projective cover of $X$. To show that $X$ is indecomposable,
 suppose $X=X_1\oplus X_2$ where $X_1$ and $X_2$ are both non-zero. Consider the projective covers $Q_1$ and $Q_2$ of $X_1$ and $X_2$ respectively and the associated exact sequences:
$$
0\rightarrow Y_1\rightarrow Q_1\rightarrow X_1\rightarrow 0,
$$ 
$$
0\rightarrow Y_2\rightarrow Q_2\rightarrow X_2\rightarrow 0.
$$ 
Then $Q_1\oplus Q_2$ is the projective cover of $X_1\oplus X_2\cong X$. Because  the projective cover of $X$ is unique up to isomorphism it follows that $Q\simeq Q_1\oplus Q_2$. Therefore we have $Y\simeq Y_1\oplus Y_2$. Since $Y$ is indecomposable, either $Y_1=0$ or $Y_2=0$. If $Y_1=0$, then $X_1\simeq Q_1\neq 0$, which contradicts the fact that $g$ is left minimal. A similar contradiction is drawn if we assume $Y_2=0$. Therefore, $X$ is indecomposable.

$(2)$ This is the dual statement of $(1).$
\end{proof}

\begin{cor} \label{ind}
Let $0\rightarrow Y\stackrel{g}\rightarrow Q\stackrel{f}\rightarrow X\rightarrow 0$ be a non-split exact sequence, where $Q$ is a projective-injective module. Then the following statements are equivalent:
\begin{enumerate}
\item $X$ is indecomposable and $f$ is a projective cover,
\item $Y$ is indecomposable and $g$ is an injective envelope. 
 \end{enumerate}
\end{cor}

Applying Corollary \ref{ind} recursively, we have the following result: 

\begin{cor}\label{ind cosyz}
Suppose  $0\rightarrow X\rightarrow I_0\stackrel{d_0}\rightarrow I_1\stackrel{d_1}\rightarrow\cdots$ is a minimal injective resolution of an indecomposable module $X$. If $I_0$, $I_1$, $\cdots$, $I_k$ are also projective, then $\Ima d_i$ are indecomposable for all $0\leq i\leq k$.
\end{cor}

%
\begin{lem}\label{cov-env}
Let $X$ be an indecomposable module in $\cx$ and let
$$
0\rightarrow P\stackrel{i}\rightarrow Q\stackrel{p}\rightarrow X\rightarrow 0
$$
be the minimal projective resolution of $X$. Then the following statements hold:
\begin{enumerate}
\item The module $Q$ is projective-injective. 
\item The syzygy $\Omega X=P$ is indecomposable, projective and non-injective.
\item The map $i: P\rightarrow Q$ is an injective envelope of $P$.
\end{enumerate}
\begin{proof}
(1) By Proposition $\ref{PinC}$, $Q$ is also injective. \\
(2) It is clear that $P$ is projective non-injective.  By Corollary \ref{ind}, $P$ is indecomposable.\\
(3) The fact that the map $i$ is an injective envelope also follows from Corollary \ref{ind}.
\end{proof}
\end{lem}


\begin{defn}
Let $\mathcal P$ be the subcategory of projective non-injective modules in $\modd\Lambda$. 
Denote by $[M]$ the isomorphism class of a $\Lambda$-module $M$. Then by Lemma $\ref{cov-env}$, we know that
 $\mathbf{\Omega}([X]):=[\Omega X]=[P]$
defines a set-theoretic map:
$\mathbf{\Omega}: \ind\cx\rightarrow \ind \mathcal P$.
\end{defn}

Now we show the main Lemma:

\begin{lem}\label{main}
Let $\Lambda$ be an artin algebra. Let $n$ be the number of non-isomorphic simple $\Lambda$-modules.
 Then 
 \begin{enumerate}
\item 
$
\mathbf{\Omega}: \ind \cx\rightarrow \ind \mathcal P
$  is an injection of sets,
\item $n_\mathcal Q+n_{\cx}\leq n$, 
\item $n_\mathcal Q+n_{\cx}= n$  if and only if $\mathbf{\Omega}$ is a bijection.
\end{enumerate}
\end{lem}

\begin{proof} (1) Injectivity of $\mathbf{\Omega}$: Suppose $X_1\ncong X_2$ in $\ind \cx$. We will show that $\mathbf{\Omega}([X_1]) \neq \mathbf{\Omega}([X_2])$. In fact, taking the minimal projective resolution of $X_1$ and $X_2$, we get 
$$
0\rightarrow P_1\rightarrow Q_1\rightarrow X_1\rightarrow 0,
$$ 
$$
0\rightarrow P_2\rightarrow Q_2\rightarrow X_2\rightarrow 0.
$$ 
Assume $P_1\cong P_2$. By Corollary $\ref{ind}$, $Q_1$ and $Q_2$ are injective envelopes of $P_1$ and $P_2$ respectively. Hence $Q_1\cong Q_2$, and then $X_1\cong X_2$ which is a contradiction. \\
(2) It follows from (1) that $n_{\cx}\leq n_\mathcal P$. Therefore $n_\mathcal Q+n_{\cx}\leq n_\mathcal Q+n_\mathcal P=n.$ Then (3) is clear.
\end{proof}
\begin{cor}\label{main cor}
Let $\Lambda$ be an artin algebra with $n$ simple modules. Then there is a tilting module $T_\cC$ in  $\mathcal C_\Lambda$ if and only if $\mathbf{\Omega}$ is a bijection.
\end{cor}

\begin{cor}\label{T_C compute}
If a tilting module $T_\cC$ exists, then $T_\cC \simeq \widetilde{Q} \oplus \left(\bigoplus_i \Omega^{-1} P_i\right)$, where $\Omega^{-1} P_i$ is the cosyzygy of $P_i$ and the direct sum is taken over representatives of the isomorphism classes of all indecomposable projective non-injective $\Lambda$-modules $P_i$.
\end{cor}

\subsection{Existence of tilting modules in $\cc _{\Lambda}$ in terms of dominant dimension}
\label{subsec:tilting}


%
%
%
%
%
%
We now prove the main theorem.
Recall that the dominant dimension of a (left) $\Lambda$-module $M$ is defined as follows.
\begin{defn} \label{dom.dim.defn} 
Let $0\rightarrow \,_\Lambda M\rightarrow I_0\rightarrow I_1\rightarrow \cdots\rightarrow I_m\rightarrow \cdots$  be a minimal injective resolution of $M$. Then $\domdim \,_\Lambda M=\sup \, \{k \mid  I_i \text{\ is\ projective, for all }  0\leq i<k\}$. The left dominant dimension of the algebra $\Lambda$ is defined to  be $\domdim \,_\Lambda\Lambda$.
\end{defn}

\begin{rmk}\label{domdim symmetry}
The dominant dimension of a right module and the right dominant dimension of the algebra are defined similarly. It is well known that $\domdim\,_\Lambda\Lambda=\domdim \Lambda_\Lambda$ for any algebra $\Lambda$  (see e.g. \cite[Theorem 4]{M}). So for the rest of this paper, we will denote both left and right dominant dimension of $\Lambda$ by $\domdim\Lambda$ and call it the \textbf{dominant dimension of $\Lambda$}. 
\end{rmk}
\begin{rem} \label{dom.dim.rem} Here are some basic properties of dominant dimension:
\begin{enumerate}
\item Let $Q$ be a projective-injective module. Then $\domdim Q=\infty$. 
\item $\domdim\Lambda=\min \, \{\domdim \,_\Lambda P \mid \,_\Lambda P\text{\ is\ indecomposable\ projective} \}$.
\item If $\Lambda$ is selfinjective, then $\domdim\Lambda=\infty$.
\end{enumerate}
\end{rem}

\begin{thm}\label{tilting}
Let $\Lambda$ be an artin algebra. Then the following statements are equivalent:
\begin{enumerate}
\item The subcategory $\mathcal C_\Lambda$ contains a tilting $\Lambda$-module $T_\cC$,
\item $\domdim \Lambda\geq2$.
\end{enumerate}
\end{thm}
\begin{proof}
$(1)\implies (2)$. 
Let $P$ be an indecomposable projective $\Lambda$-module. \\
If $P$ is projective-injective then $\domdim P=\infty$ by Remark \ref{dom.dim.rem}.\\
If $P$ is projective non-injective we will show that in the minimal injective copresentation of $P$ 
$$0\rightarrow P\rightarrow I_0\rightarrow I_1\rightarrow I_2\rightarrow I_3 \rightarrow  \dots, $$
both $I_0$ and $I_1$ are projective-injective. To show this we use the assumption that there is a tilting module $T_\cC$ in $\cC_{\Lambda}$. By property (3) in Definition \ref{tilt} of tilting modules, for each projective $P$ there is an exact sequence
$$0\rightarrow P\xrightarrow{g} T_0\xrightarrow{f} T_1\rightarrow 0,$$
where $T_0$ and $T_1$ are in $\add T_\cC$ and the map $g$ is a minimal left $\add T_\cC$-approximation of $P$. It follows by Proposition \ref{TapproxP} that $T_0$ is projective-injective. Call it $Q_0$. Let $T_1\xrightarrow{i} Q_1$ be the injective envelope of $T_1$.
Since $T_1$ is in $\add T_\cC \subset \cC_{\Lambda}$, the injective envelope $Q_1$ is also projective-injective by Proposition \ref{PinC}(2). Combining these two statements, we get the minimal injective copresentation of $P$
$$0\rightarrow P\xrightarrow{g} Q_0\xrightarrow{if} Q_1,$$
where $Q_0$ and $Q_1$ are projective-injective modules. Hence, $\domdim P\geq 2$. By Remark \ref{dom.dim.rem}, we have $\domdim \Lambda \geq 2$.

$(2)\implies (1)$. 
We use the fact that $\domdim \Lambda \geq 2$ in order to show that the map 
$$\mathbf{\Omega}: \ind\cx\rightarrow \ind \mathcal P$$ 
is a bijection. Then apply Corollary \ref{main cor} to conclude (1). It follows from Lemma \ref{main} that $\mathbf{\Omega}$ is an injective map. To show that it is surjective, we consider $P \in \ind \mathcal P$ and find $X \in \ind\cx$ so that $\Omega X\cong P$. Let $P\xrightarrow{a} Q$ be the injective envelope of $P$. The module $Q$ is projective-injective since $\domdim P\geq 2$. Consider the induced short exact sequence
$$0\rightarrow P\xrightarrow{a} Q\xrightarrow{b} X\xrightarrow {}0.$$
Let $X\xrightarrow{c}I$ be the injective envelope of $X$. We have a minimal injective copresentation of $P$:
$$0\rightarrow P\xrightarrow{a} Q\xrightarrow{cb} I.$$
The assumption that $\domdim P\geq 2$ implies that $I$ is projective-injective. Therefore $X$ is a submodule of a projective-injective module. Since $X$ is also a quotient of $Q$ and $\pd X=1$, it follows that $X$ is in $\cx$. Furthermore, $X$ is indecomposable since $P$ is indecomposable, by Corollary \ref{ind}. Therefore $X$ is in $\ind\cx$ and $\Omega X\cong P$.
Thus $\mathbf{\Omega}$ is a surjection and therefore a bijection. By Corollary \ref{main cor}, it follows that there is a tilting module $T_\cC$ in $\cC_{\Lambda}$.
\end{proof}
%
%
%
Using this theorem, we can deduce the following result of Crawley-Boevey and Sauter \cite{CBS}.
\begin{cor}
If $\gldim \Lambda=2$, then $\mathcal C_\Lambda$ contains a tilting $\Lambda$-module if and only if $\Lambda$ is an Auslander algebra.
\end{cor}
\begin{proof} An Auslander algebra is an algebra $A$ with $\gldim A=2$ and $\domdim A=2$.
\end{proof}


More directly by Theorem $\ref{tilting}$, for $m$-Auslander algebras $\Lambda$, we can always guarantee the existence of such a tilting module in $\cC_\Lambda$. Recall that Iyama introduced the notion of higher Auslander algebras (see \cite[2.2]{I2}): an artin algebra $\Lambda$ is called {\bf $m$-Auslander} if $\gldim\Lambda\leq m+1\leq \domdim\Lambda$. It is easy to see that $m$-Auslander algebras are either semisimple or satisfy $\gldim\Lambda=\domdim\Lambda$. Notice that Auslander algebras are precisely $1$-Auslander algebras.

\begin{cor}\label{m-Aus}
For any integer $m \geq 1$ and any $m$-Auslander algebra $\Lambda$, its subcategory $\mathcal C_\Lambda$ always contains a tilting $\Lambda$-module. 
\end{cor}

\begin{exm}
In this example, we illustrate Corollary \ref{m-Aus} for a $2$-Auslander algebra. Let $\Lambda$ be the Nakayama algebra given by the following quiver and relations $\alpha\gamma=\gamma\beta=0$. \\

$ \quad
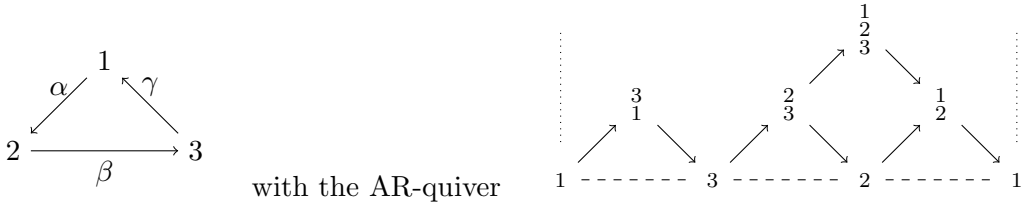
\begin{tikzpicture}[->]
\node(1) at (0,0.2) {$1$};
\node(2) at (-1.2, -1) {$2$};
\node(3) at (1.2,-1) {$3$};

\draw(1)--node [above]{$\alpha$}(2);
\draw(2)--node [below]{$\beta$}(3);
\draw(3)--node [above]{$\gamma$}(1);
\end{tikzpicture} 
$ 
\quad with the AR-quiver \quad  
$
\begin{tikzpicture}[->][scale=.75]
\node(s1) at (0,0) {$\begin{smallmatrix} 1 \end{smallmatrix}$};
\node(s3) at (2,0) {$\begin{smallmatrix} 3 \end{smallmatrix}$}; 
\node(s2) at (4,0) {$\begin{smallmatrix} 2 \end{smallmatrix}$};
\node(s1') at (6,0) {$\begin{smallmatrix} 1 \end{smallmatrix}$};
\node(p3) at (1,1) {$\begin{smallmatrix} 3\\1 \end{smallmatrix}$};
\node(p2) at (3,1) {$\begin{smallmatrix} 2\\3 \end{smallmatrix}$};
\node(i2) at (5,1) {$\begin{smallmatrix} 1\\2 \end{smallmatrix}$}; 
\node(p1) at (4,2) {$\begin{smallmatrix} 1\\2\\3 \end{smallmatrix}$};
 
\draw (s1)--(p3);
\draw (p3)--(s3);
\draw (s3)--(p2);
\draw (p2)--(p1);
\draw (p2)--(s2);
\draw (s2)--(i2);
\draw (p1)--(i2);
\draw (i2)--(s1');
\draw[dashed][-](s1)--(s3);
\draw[dashed][-](s3)--(s2);
\draw[dashed][-](s2)--(s1');

\draw[dotted][-](0,0.5)--(0,2);
\draw[dotted][-](6,0.5)--(6,2);
\end{tikzpicture}
$ \\

\noindent
The subcategory $\mathcal C_\Lambda$ is $\add\{ \begin{smallmatrix} 3\\1 \end{smallmatrix}, \begin{smallmatrix} 1\\2\\3 \end{smallmatrix}, \begin{smallmatrix} 1 \end{smallmatrix}, \begin{smallmatrix} 3 \end{smallmatrix} \}$, where there is a tilting $\Lambda$-module $T_\cC=\begin{smallmatrix} 3\\1 \end{smallmatrix}\oplus \begin{smallmatrix} 1\\2\\3 \end{smallmatrix} \oplus\begin{smallmatrix} 1 \end{smallmatrix}$.
\end{exm}


%


\subsection{Existence of cotilting modules in $\cc _{\Lambda}$}
\label{subsec:cotilting}

 Notice that a left $\Lambda$-module $ _\Lambda T$ is tilting if and only if $D(T)_\Lambda$ as a right $\Lambda$-module is cotilting.  As a dual statement, we provide a result on the existence of a cotilting module here.

%
\begin{thm}\label{cotilting cor}
Let $\Lambda$ be an artin algebra. Then the following statements are equivalent:
\begin{enumerate}
\item $\mathcal C_{\Lambda}$ contains a cotilting $\Lambda$-module,
\item $\mathcal C_{\Lambda^{op}}$ contains a tilting $\Lambda^{op}$-module,
\item $\domdim \Lambda^{op}\geq 2$.
\end{enumerate}
\end{thm}

On the other hand, by definition, we have $\domdim\Lambda^{op}=\domdim\Lambda_\Lambda$ which is the same as $\domdim\Lambda$ as we mentioned before (see Remark \ref{domdim symmetry}). So combining our results, we have:
\begin{cor}\label{cotilting cor}
Let $\Lambda$ be an artin algebra. Then the following statements are equivalent:
\begin{enumerate}
\item $\domdim \Lambda\geq 2$,
\item $\mathcal C_\Lambda$ contains a tilting $\Lambda$-module $T_\cC$,
\item $\mathcal C_\Lambda$ contains a cotilting $\Lambda$-module $C_\cC$. 
\end{enumerate}
\end{cor}

\begin{rmk}\label{C_C compute}
 If a cotilting module $C_\cC$ exists, then $C_\cC \simeq \widetilde{Q} \oplus \left(\bigoplus_i \Omega I_i\right)$, where $\Omega I_i$ is the syzygy of $I_i$ and the direct sum is taken over the representatives of the isomorphism classes of all indecomposable injective non-projective $\Lambda$-modules $I_i$.
\end{rmk}

\begin{rem}
By \cite[Proposition 2.6]{PS}, for $k\geq1$, the existence of the canonical $k$-tilting (or $k$-cotilting) modules is equivalent to the dominant dimension of the algebra being at least $k$. One can see that the implications $(1)\implies(2)$ and $(1)\implies(3)$ in Corollary~\ref{cotilting cor} follow immediately from the existence of the canonical $1$-tilting (or $1$-cotilting) modules.  However, our proof of the equivalence of statements $(1),(2),(3)$ is done using a different approach.
\end{rem}

In general, the tilting module and the cotilting module in $\mathcal C_\Lambda$ from Corollary~\ref{cotilting cor} are not necessarily the same module. Now we discuss when $\mathcal C_\Lambda$ contains a module which is \emph{both} tilting and cotilting. We call this module the tilting-cotilting module in $\mathcal C_\Lambda$.

\begin{defn}\cite{AR1} An artin algebra $\Lambda$ is called \textbf{Gorenstein} if both $\id _{\Lambda}\Lambda<\infty$ and $\id\Lambda_{\Lambda}<\infty$.
\end{defn}

\begin{rmk}\label{gdim defn}
It is conjectured that for an artin algebra $\Lambda$, $\id _{\Lambda}\Lambda<\infty$ is equivalent to $\id\Lambda_{\Lambda}<\infty$ (Gorenstein Symmetry Conjecture). But we know that if $\id _{\Lambda}\Lambda$ and $\id\Lambda_{\Lambda}$ are both finite, then $\id _{\Lambda}\Lambda=\id\Lambda_{\Lambda}$ (e.g. see \cite{Z}); in this case, we call this number \textbf{Gorenstein dimension}, denoted as $\Gdim \Lambda :=\id _{\Lambda}\Lambda=\id\Lambda_{\Lambda}$. An artin Gorenstein algebra $\Lambda$ is called \textbf{Iwanaga-Gorenstein of Gorenstein dimension $m$}, if $\Gdim\Lambda=m$.  To avoid confusion, we point out that in the literature, there is an original notion of $m$-Gorenstein algebra (e.g. see \cite{AR2}) which is different from the notion of Iwanaga-Gorenstein algebra of Gorenstein dimension $m$.
\end{rmk}

It is well known that selfinjective algebras and algebras of finite global dimensions are Gorenstein. For the convenience of the readers, we show:
\begin{prop}\label{gldim finite gor}
Let $\Lambda$ be an artin algebra with $\gldim\Lambda=d<\infty$. Then there exists an indecomposable projective $\Lambda$-module $P$ with $\id P = d$. 

That is, if $\gldim\Lambda=d<\infty$ then $\Lambda$ is Iwanaga-Gorenstein with $\Gdim\Lambda=d$.
\end{prop}
\begin{proof}
Since $\gldim\Lambda=d$, $\id P\leq d$ for all projective $\Lambda$-module $P$. Here, we claim that at least one $P$ satisfies $\id P=d$. Otherwise, since any $\Lambda$-module $M$ has a finite projective resolution $0\rightarrow P_d\rightarrow \cdots\rightarrow P_0\rightarrow M\rightarrow 0$, with each $\id P_i<d$, then $\id M<d$ and hence $\gldim \Lambda<d$, a contradiction. 
So $\gldim \Lambda=d$ implies $\id_\Lambda \Lambda=d$. 
Similarly, we have $\pd D(\Lambda_\Lambda)=d=\id\Lambda_\Lambda$. Therefore, $\Lambda$ is Iwanaga-Gorenstein with $\Gdim\Lambda=d$.
\end{proof}

Recently Iyama and Solberg defined $m$-Auslander-Gorenstein algebra in \cite{IS}. They also showed that the notion of $m$-Auslander-Gorenstein algebra is left and right symmetric.

\begin{defn}\label{mAusGor} \cite{IS} An artin algebra $\Lambda$ is called \textbf{$m$-Auslander-Gorenstein}\footnote{The authors called it ``minimal $m$-Auslander-Gorenstein'' in the introduction of \cite{IS}. Later in the paper, they called it ``$m$-Auslander-Gorenstein'' for simplicity. } if $$\id\,_\Lambda\Lambda\leq m+1\leq\domdim\Lambda.$$  
\end{defn}

\begin{prop}\cite[Proposition 4.1]{IS} Let $\Lambda$ be an artin algebra. 
\begin{enumerate}
\item If $\Lambda$ is an $m$-Auslander-Gorenstein algebra, then either $\id_\Lambda\Lambda =m+1=\domdim \Lambda$ holds or $\Lambda$ is selfinjective. 
\item An algebra $\Lambda$ is $m$-Auslander-Gorenstein if and only if $\Lambda^{op}$ is $m$-Auslander-Gorenstein.
\end{enumerate}
\end{prop}

\begin{rmk}\label{mAusGor}
Let $\Lambda$ be an artin algebra. We have the following equivalent characterizations of $m$-Auslander-Gorenstein algebras: 
\begin{enumerate}
\item $\Lambda$ is $m$-Auslander-Gorenstein,
\item $\Lambda$ is Iwanaga-Gorenstein with $\Gdim \Lambda \leq m+1 \leq \domdim\Lambda$,
\item $\Lambda$ is selfinjective or $\id\,_\Lambda\Lambda=\id\Lambda_\Lambda=m+1=\domdim\Lambda$,
\item $\Lambda$ is selfinjective or $\id\,_\Lambda\Lambda=m+1=\domdim\Lambda$,
\item $\Lambda$ is selfinjective or $\id\Lambda_\Lambda=m+1=\domdim\Lambda$,
\item $\id\,_\Lambda\Lambda\leq m+1\leq\domdim\Lambda$,
\item $\id\Lambda_\Lambda\leq m+1\leq\domdim\Lambda$.
\end{enumerate}
\end{rmk}

In particular, a $1$-Auslander-Gorenstein algebra is either a selfinjective algebra or a Gorenstein algebra satisfying $\id\,_\Lambda\Lambda=2=\domdim \Lambda$. 

\begin{rmk}\label{dtr selfinj}
The algebras satisfying the condition $\id\,_\Lambda\Lambda=\domdim \Lambda=2$ are called {\bf $DTr$-selfinjective algebras} and they were classified by Auslander and Solberg in \cite{AS}.
\end{rmk}

We have a characterization of $1$-Auslander-Gorenstein algebras in terms of the existence of the tilting-cotilting module in $\cC_\Lambda$:

\begin{thm}\label{tilting-cotilting}
Let $\Lambda$ be an artin algebra. Then the following statements are equivalent: 
\begin{enumerate}
\item $\Lambda$ is $1$-Auslander-Gorenstein,
\item $\cC_\Lambda$ contains a tilting-cotilting module.
\end{enumerate}
\end{thm}
\begin{proof}
{ $(1) \implies (2)$.} Assume that $(1)$ holds, then $\domdim\Lambda\geq2$. By Corollary $\ref{cotilting cor}$, the subcategory $\mathcal C_\Lambda$ contains a tilting module $T_\cC$. It suffices to show that $\id T_\cC\leq 1$. 
 
In fact, let $T_0$ be any non-injective indecomposable summand of $T_\cC$ (if it exists). Then by Proposition $\ref{PinC}$, $T_0$ is neither projective nor injective with $\pd T_0= 1$. Moreover, $T_0$ has a minimal projective resolution:
$$
0\rightarrow P_1\rightarrow P_0\rightarrow T_0\rightarrow 0,
$$
where $P_0$ is projective-injective. Because $\domdim \Lambda\geq2$, $T_0$ is a submodule of a projective-injective module $I_0$. But $I_0/T_0$ must be injective since $\id_\Lambda\Lambda\leq 2$. Hence $\id T_0=1$. Therefore $\id T_\cC\leq1$ and $T_\cC$ is a cotilting module.

{ $(2) \implies (1)$.} Assume that $(2)$ holds, then Corollary $\ref{cotilting cor}$ implies that $\domdim \Lambda\geq2$. Let $P$ be an indecomposable projective non-injective module (if it exists). Let $f: P\rightarrow I(P)$ be an injective envelope of $P$. Then we know that $X\cong \Coker f$ is a non-injective summand of the tilting module $T_\cC$. Since $T_\cC$ is also cotilting, we have that $\id X=1$, which implies $\id P=2$. Hence, we show $\id\,_\Lambda \Lambda\leq2$. 

Therefore $\id\,_\Lambda \Lambda\leq2\leq\domdim\Lambda$ which means that $\Lambda$ is $1$-Auslander-Gorenstein. 
 \end{proof}

We have the following statement which generalizes Crawley-Boevey and Sauter's result \cite[Lemma 1.1]{CBS} from an algebra $\Lambda$ with $\gldim \Lambda = 2$ to an algebra $\Lambda$ with any finite $\gldim \Lambda$.

\begin{cor}\label{tilting-cotilting cor}
Let $\Lambda$ be an artin algebra with $\gldim \Lambda<\infty$. Then the subcategory $\mathcal C_\Lambda$ contains a tilting-cotilting module if and only if $\Lambda$ is an Auslander algebra.
\end{cor}
\begin{proof}
Assume $\mathcal C_\Lambda$ contains a tilting-cotilting module, then $\Lambda$ is $1$-Auslander-Gorenstein by Theorem \ref{tilting-cotilting}. It follows from Remark \ref{mAusGor} that $\Lambda$ is Iwanaga-Gorenstein with $\Gdim \Lambda = 2$. 
By Proposition \ref{gldim finite gor}, this forces $\gldim\Lambda=2$.
Also $\domdim \Lambda\geq 2$ follows from Theorem $\ref{tilting}$. Thus, $\Lambda$ is an Auslander algebra. The converse holds by using Theorem $\ref{tilting-cotilting}$ and the fact that an Auslander algebra is $1$-Auslander-Gorenstein.
\end{proof}

\begin{rmk} \
\begin{itemize}
 \item Theorem \ref{tilting-cotilting} and Corollary $\ref{tilting-cotilting cor}$ suggest that $1$-Auslander-Gorenstein algebras are generalizations of Auslander algebras in the sense of the existence of a tilting-cotilting module in $\cC_\Lambda$.

 \item A more general situation is considered by Pressland and Sauter (c.f.~\cite[Proposition 3.7, Theorem 3.9]{PS}). They show that $\Lambda$ is an $m$-Auslander-Gorenstein algebra if and only if canonical $k$-tilting modules coincide with canonical $(m+1-k)$-cotilting modules, for all $0 \leq k \leq m+1$.
\end{itemize}
\end{rmk}

\begin{exm}
In this example, we present a $1$-Auslander-Gorenstein algebra which contains a tilting-cotilting module in $\cC_\Lambda$ but is not an Auslander algebra, since its $\gldim \Lambda = \infty$. Let $\Lambda$ be the Nakayama algebra given by the following quiver and relations $\gamma\beta\alpha=\delta\gamma\beta=\alpha\epsilon=\epsilon\delta=0$. We omit the modules when drawing the AR-quiver. \\

$ \quad
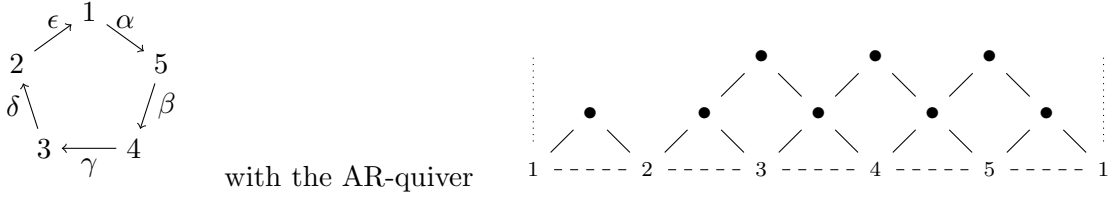
\begin{tikzpicture}[->][scale=.9]
\node(1) at (0,1.81) {$1$};
\node (2) at (0.95, 1.12) {$5$};
\node(3) at (0.59,0) {$4$};
\node (4) at (-0.59,0) {$3$};
\node (5) at (-0.95,1.12) {$2$};

\draw(1)--node [above]{$\alpha$}(2);
\draw(2)--node [right]{$\beta$}(3);
\draw(3)--node [below]{$\gamma$}(4);
\draw(4)--node [left]{$\delta$}(5);
\draw(5)--node [above]{$\epsilon$}(1);
\end{tikzpicture} 
$
\quad with the AR-quiver \quad
$
\begin{tikzpicture}[scale=.75]
\node(s1) at (0,0) {$\begin{smallmatrix} 1 \end{smallmatrix}$};
\node(s2) at (2,0) {$\begin{smallmatrix} 2 \end{smallmatrix}$}; 
\node(s3) at (4,0) {$\begin{smallmatrix} 3 \end{smallmatrix}$};
\node(s4) at (6,0) {$\begin{smallmatrix} 4 \end{smallmatrix}$};
\node(s5) at (8,0) {$\begin{smallmatrix} 5 \end{smallmatrix}$};
\node(s1') at (10,0) {$\begin{smallmatrix} 1 \end{smallmatrix}$};

\node(12) at (1,1) {$\bullet$};
\node(22) at (3,1) {$\bullet$};
\node(32) at (5,1){$\bullet$};
\node (42) at (7,1){$\bullet$};
\node (52) at (9,1){$\bullet$};

\node(23)at (4,2){$\bullet$};
\node(33) at (6,2){$\bullet$};
\node (43)at (8,2){$\bullet$};

\draw[dashed][-](s1)--(s2);
\draw[dashed][-](s2)--(s3);
\draw[dashed][-](s3)--(s4);
\draw[dashed][-](s4)--(s5);
\draw[dashed][-](s5)--(s1');

\draw(s1)--(12);
\draw(s2)--(22)--(23);
\draw(s3)--(32)--(33);
\draw(s4)--(42)--(43);
\draw(s5)--(52);

\draw(12)--(s2);
\draw(22)--(s3);
\draw(23)--(32)--(s4);
\draw(33)--(42)--(s5);
\draw(43)--(52)--(s1');

\draw[dotted][-](0,0.5)--(0,2);
\draw[dotted][-](10,0.5)--(10,2);
\end{tikzpicture}
$ \\

\noindent
There exists a tilting-cotilting module $T_\cC=P_1\oplus P_2\oplus P_4\oplus P_5\oplus S_4$ in $\mathcal C_\Lambda$.
\end{exm}

\section{Homological applications of the tilting module in $\cC_\Lambda$}
\subsection{More on dominant dimension}\label{DomDim}

Algebras with dominant dimension at least $2$ have been studied since the 1960's by Morita \cite{Mo}, Tachikawa \cite{T}, Mueller \cite{M}, Ringel \cite{R2} and many others. Morita and Tachikawa showed that any artin algebra of dominant dimension at least $2$ is an endomorphism algebra of a generator-cogenerator of another artin algebra. This gives us a full machinery for producing algebras whose dominant dimension is at least $2$, and hence, algebras which have a tilting module in $\cC_\Lambda$.

\begin{thm}\cite{Mo}, \cite{T}, \cite[Theorem 2]{M}
\label{Morita-Tachikawa}
For an artin algebra $\Gamma$, the following statements are equivalent:
\begin{enumerate}
\item $\domdim \Gamma\geq2$,
\item $\Gamma\simeq \End_\Lambda(X)^{op}$, where $X$ is a generator-cogenerator of an artin algebra $\Lambda$.
\end{enumerate}
\end{thm}

Furthermore, there is a more precise result on the dominant dimension:

\begin{lem}\cite[Lemma 3]{M}\label{domdim ext}
Let $\Gamma\simeq \End_\Lambda(X)^{op}$, where $X$ is a generator-cogenerator of an artin algebra $\Lambda$. Then $\domdim\Gamma\geq m+2$ if and only if $\Ext^i_\Lambda(X,X)=0$, for all $1\leq i \leq m$, $m \in \{0,1,2,\ldots\}$.
\end{lem}

From this lemma, the following well-known results can be deduced. 

\begin{cor}\label{domdim sum}
Let $X$ be a generator-cogenerator of an artin algebra $\Lambda$ and $\Gamma \simeq \End_\Lambda(X)^{op}$. 
\begin{enumerate}
\item  If $X$ is a summand of a module $Y$, then $\domdim \, \End_\Lambda(X)^{op}\geq \domdim \, \End_\Lambda(Y)^{op}$.
\item Suppose $\Lambda$ is non-selfinjective. If $\id_\Lambda\Lambda\leq m$ or $\id \Lambda_\Lambda\leq m$, then \\
$\domdim\Gamma\leq m+1$. 
 \item If $\gldim \Lambda\leq m$, then $\domdim\Gamma\leq m+1$.
\item If $\Lambda$ is non-semisimple hereditary, then $\domdim\Gamma=2$.

\end{enumerate}
\end{cor}

Next, we recall how to construct the algebra $\Lambda$ and $\Lambda$-module $X$ such that $\Gamma \simeq \End_\Lambda(X)^{op}$ is of dominant dimension at least $2$, under the assumption that both $\Gamma$ and $\Lambda$ are basic. In general, we emphasize that such an algebra $\Lambda$ is only unique up to Morita equivalence, see also \cite[IV]{A1}, \cite{R2} for details.

\begin{prop}\label{end mod}
Let 
$X$ be a module over an artin algebra $\Lambda$ and $\Gamma:=\End_\Lambda(X)^{op}$. Then:
\begin{enumerate}
\item $\Hom_\Lambda(X,X_i)$ are all the indecomposable projective $\Gamma$-modules, where $X_i$ runs through the non-isomorphic indecomposable direct summands of $X$.
\item If $X$ is a cogenerator of $\Lambda$, then the $\Hom_\Lambda(X,I_i)$ are all the indecomposable projective-injective $\Gamma$-modules, where $I_i$ runs through the non-isomorphic indecomposable injective $\Lambda$-modules.
\end{enumerate}
\end{prop}

\begin{prop}\label{morita Q}
Let $\Gamma$ be a basic algebra of dominant dimension at least $2$. Let $\widetilde{Q}$ be the direct sum of representatives of the isomorphism classes of all indecomposable projective-injective $\Gamma$-modules. Then the algebra $\Lambda$ and $\Lambda$-module $X$ can be chosen to be $\Lambda=\End_\Gamma(\widetilde Q)^{op}$ and $X=\Hom_{\Gamma}(\widetilde Q, D\Gamma)$ so that $\Gamma \simeq \End_\Lambda(X)^{op}$.
\end{prop}
 
%
 
We now describe the tilting module $T_\cC$ (whose existence is given by Theorem \ref{tilting}) in terms of the algebra $\Lambda$ and a generator-cogenerator $X$.
 
\begin{prop}\label{XT}
Let $\Gamma$ be a basic algebra with $\domdim\Gamma\geq2$. Let $T_\cC$ be the tilting module in $\cC_\Gamma$. 
Let $\Gamma \simeq \End_\Lambda(X)^{op}$, for an artin algebra $\Lambda$ and a $\Lambda$-generator-cogenerator $X$.  
 Then:
\begin{enumerate}
 \item $T_\cC \simeq \left(\bigoplus_i \Hom_\Lambda(X,I_i) \right) \oplus \left(\bigoplus_j \Hom_\Lambda(X,I_0(X_j)/X_j) \right)$, where $\{I_i\}$ are the indecomposable injective $\Lambda$-modules, $\{X_j\}$ are the non-injective indecomposable direct summands of $X$, and $\{I_0(X_j)\}$ are the corresponding injective envelopes of $\{X_j\}$. 

 \item In part (1), the first summand is a projective-injective $\Gamma$-module isomorphic to $\widetilde{Q}$, and the second summand is a non-projective-injective $\Gamma$-module isomorphic to $\left(\bigoplus_i \Omega^{-1} P_i\right)$, as described in Corollary~\ref{T_C compute}.
\end{enumerate}

\end{prop}
 
%
%
%
%

\subsection{The endomorphism algebra of the tilting module}
\label{endomorphism algebra}

Using the Morita-Tachikawa correspondence as in Theorem~\ref{Morita-Tachikawa}, for any artin algebra $\Lambda$, taking a generator-cogenerator $X\in \modd \Lambda$, the algebra $\Gamma=\End_{\Lambda}(X)^{op}$ is an artin algebra of dominant dimension at least $2$. So by Theorem~\ref{tilting}, there exists a unique tilting module $T_\cC \in \cC_\Gamma$. We are going to study the global dimension of the endomorphism algebra $B_\cC:=\End_{\Gamma}(T_\cC)^{op}$ and in the next section, we obtain its relationship with the Finitistic Dimension Conjecture. 

We recall some facts about tilting modules and torsion classes due to Brenner and Butler (see \cite[VI, \S3, \S4]{ASS} \label{BBtilting} for more details and proofs). Let $\Lambda$ be an artin algebra and $T$ be any tilting $\Lambda$-module. Then there is a torsion pair $(\cT(T),\cF(T))$ in $\modd \Lambda$:
$$\cT(T):=\Gen_\Lambda(T)=\{M \in \modd \Lambda \mid \Ext_\Lambda^1(T,M)=0\} \text{ and} \ \ \ \ \  (*)$$ 
$$\cF(T):=\Sub_\Lambda(\tau T)=\{M \in \modd \Lambda \mid \Hom_\Lambda(T,M)=0 \}. \ \ \ \ \ \  (**)$$ 
Let $B:=\End_\Lambda(T)^{op}$, then $T$ is also a tilting $B^{op}$-module. We have a torsion pair $(\cX(T),\cY(T))$:
$$\cX(T):=D\Gen_{B^{op}}(T)=\{M \in \modd B \mid \Hom_B(M,DT)=0\} \text{ and} $$ 
$$\cY(T):=D\Sub_{B^{op}}(\tau T)=\{M\in \modd B \mid \Ext^1_B(M,DT)=0\}.$$ 
In all four descriptions above, the first equalities may be considered as definitions and the second equalities are consequences of the results stated in \cite[VI, \S3, \S4]{ASS}.
\begin{thm}[Brenner-Butler Tilting Theorem] \cite[VI, \S3]{ASS} \label{BBtilting}

$(a)$ The functors $\Hom_\Lambda(T,-)$ and $-\otimes_B T$ induce quasi-inverse equivalences:  
$$\cT(T)\ \  {{\xrightarrow{\ \ \Hom_\Lambda(T,-)
}}\atop \overleftarrow{\ \ \ \ -\scriptstyle{\otimes_B T\ \ \ 
 } }}\ \ \cY(T).$$

$(b)$ The functors $\Ext_\Lambda^1(T,-)$ and $\Tor^B_1(-,T)$ induce quasi-inverse equivalences: 
$$\cF(T)\ \  {{\ \ \xrightarrow{\ \ \Ext^1_\Lambda(T,-)\ \ }}\atop \overleftarrow{\ \ \scriptstyle{ \ \ \Tor^B_1(-,T)\  
} }}\ \ \cX(T).$$
\end{thm}

We will first state and prove several general lemmas about tilting modules.

\begin{lem} \label{syzygy} Let $\Lambda$ be an artin algebra, $T$ be any tilting $\Lambda$-module, and $B:=\End_\Lambda(T)^{op}$. 
\begin{enumerate}
\item  Let $U$ be a $B$-module. Then the first syzygy $\Omega U$ of $U$ is in $\cY(T)$.
\item There exists a $\Lambda$-module $M\in \cT(T)$  such that $\Omega U \cong \Hom_\Lambda(T,M)$.
\end{enumerate}
\end{lem}
\begin{proof} Let $U$ be a $B$-module and $\Hom_\Lambda(T,T_0) \rightarrow U$ be its projective cover. In the exact sequence
$
0\rightarrow \Omega U \rightarrow \Hom_\Lambda(T,T_0) \rightarrow U\rightarrow 0
$, 
the middle term $\Hom_\Lambda(T,T_0)$ is in $\cY(T)$ by Theorem \ref{BBtilting}(a), and $\cY(T)$ is a torsion-free class so it is closed under submodules. Therefore, the first syzygy $\Omega U$ of $U$ is in $\cY(T)$. So $\Omega U \cong \Hom_\Lambda(T,M)$, for some $M\in \cT(T)=\Gen(T)$.
\end{proof}

\begin{lem}\label{projdim general}
Let $\Lambda$ be an artin algebra, $T$ be any tilting $\Lambda$-module, and $B:=\End_\Lambda(T)^{op}$. 
\begin{enumerate}
\item Let $M$ be a $\Lambda$-module where $M \in \cT(T)$. Then $\pd_B (\Hom_\Lambda(T,M)) \leq \pd_\Lambda M$.
\item Let $U$ be a non-projective $B$-module and $M$ be a $\Lambda$-module such that $\Omega U \cong \Hom_\Lambda(T,M)$. Then $\pd_BU \leq \pd_\Lambda M+1$. 
\item $\gldim B\leq \gldim\Lambda +1$.
\item $\gldim \Lambda\leq \gldim B +1$. 
\item $|\gldim B-\gldim\Lambda|\leq 1$.
\end{enumerate}
\end{lem}
\begin{proof} (1) follows from \cite[VI, Lemma 4.1]{ASS}.\\
(2) $\pd_{B} U\leq
\pd_{B}\Omega U+1=\pd_{B} (\Hom_\Lambda(T,M))+1 \leq \pd_\Lambda M+1$.\\
(3) This is a consequence of (2).\\
(4) Since ${}_BT$ is a tilting $B$-module and $\Lambda\cong \End_B(T)^{op}$, it follows from (3) that $\gldim \Lambda\leq \gldim B +1$. Finally, (5) is just a combination of (3) and (4).
\end{proof}

Now for our unique tilting module $T_\cC$ in $\cC_\Gamma$, we have  more precise results: \ref{projdim key} and \ref{gldim B}.
\begin{lem}\label{projdim key}
Let $\Gamma$ be an artin algebra with $\domdim \Gamma \geq 2$, 
$T_\cC$ be the unique tilting module in $\cC_\Gamma$, and $B_\cC=\End_\Gamma(T_\cC)^{op}$. If $M\in \Gen(T_\cC)$ and $\pd_\Gamma M\geq1$, then 
$$\pd_{B_\cC} (\Hom_\Gamma(T_\cC,M))= (\pd_\Gamma M) -1.$$
\end{lem}

\begin{proof}
We use induction on $\pd_\Gamma M$. First, assume $\pd_\Gamma M=1$. Since $M\in \Gen(T_\cC)=\Gen(\widetilde{Q})$, we know there is some $Q_0 \in \add \widetilde{Q}$ such that $Q_0 \rightarrow M$ is an epimorphism. Since $\pd_\Gamma M=1$, $M$ has a projective resolution: 
$$0\rightarrow P\rightarrow Q_0\rightarrow M\rightarrow 0.$$ 
Since $P \in \add \Gamma$ and $\domdim \Gamma \geq 2$, the module $M$ is a submodule of a projective-injective $\Gamma$-module $Q_1$. Therefore $M$ is in $\mathcal C_{\Gamma}$. Additionally, as $\pd_\Gamma M=1$ by assumption, it follows from Proposition \ref{rigid} that 
$M \in \add T_\cC$. Hence $\pd_{B_\cC} (\Hom_\Gamma(T_\cC,M))=0.$

Now assume $\pd_\Gamma M=d>1$. In particular, $M \notin \add T_\cC$. There is an exact sequence:
$$
0\rightarrow L\rightarrow T_0\stackrel{f}\rightarrow M\rightarrow 0,
$$
where $f$ is a right $\add T_\cC$-approximation and $L\neq 0$. This induces an exact sequence of $B_\cC$-modules:
$$
0\rightarrow \Hom_\Gamma(T_\cC,L)\rightarrow \Hom_\Gamma(T_\cC,T_0) \rightarrow \Hom_\Gamma(T_\cC,M)\rightarrow 0.
$$
It follows that, $\Ext^1_\Gamma(T_\cC,L)=0$, which implies $L\in \Gen(T_\cC)$ by (*).

Since $\pd_\Gamma T_0\leq 1$ and $\pd_\Gamma M=d>1$, it follows that $\pd_\Gamma L\leq d-1$ by the standard arguments.
By induction hypothesis, $\pd_{B_\cC} (\Hom_\Gamma(T_\cC,L))=d-2$. Therefore, due to the above exact sequence, $\pd_{B_\cC} (\Hom_\Gamma(T_\cC,M))=d-1$.
\end{proof}

\begin{cor}\label{gldim B} 
Let $\Gamma$ be an artin algebra with $\domdim \Gamma \geq 2$, 
$T_\cC$ be the unique tilting module in $\cC_\Gamma$, and $B_\cC=\End_\Gamma(T_\cC)^{op}$. 
Then $\gldim B_\cC \leq \gldim\Gamma$.
\end{cor}

\begin{proof}
Let $U$ be a $B_\cC$-module. Then by Lemma \ref{syzygy}(2) and Lemma \ref{projdim key}, it follows that \\
$\pd_{B_\cC} U \leq \pd_{B_\cC} (\Hom_\Gamma(T_\cC,M))+1 = (\pd_\Gamma M-1)+1 \leq \gldim\Gamma.$
\end{proof}

\begin{rmk} Let $\Gamma$, $T_\cC$, and $B_\cC$ be as in Corollary~\ref{gldim B}.
\begin{enumerate}
\item The sharp inequality $\gldim B_\cC < \gldim\Gamma$ does not always hold, see the Example $\ref{sharp}$.

\item 

By Lemma~\ref{projdim general}(5), either $\gldim B_\cC =\gldim \Gamma$ or $\gldim B_\cC =\gldim \Gamma - 1$.  
\end{enumerate}
\end{rmk}
\begin{exm}\label{sharp}
This is an example of $\Gamma$ such that $\gldim B_\cC =\gldim\Gamma$. Let $\Gamma$ be the Nakayama algebra given by the following quiver and relations $\gamma\beta\alpha=\delta\gamma\beta=\alpha\epsilon=0$. We omit the modules when drawing the AR-quiver. \\

$\ \ 
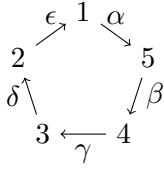
\begin{tikzpicture}[->,scale=.9]
\node(1) at (0,1.81) {$1$};
\node (2) at (0.95, 1.12) {$5$};
\node(3) at (0.59,0) {$4$};
\node (4) at (-0.59,0) {$3$};
\node (5) at (-0.95,1.12) {$2$};

\draw(1)--node [above]{$\alpha$}(2);
\draw(2)--node [right]{$\beta$}(3);
\draw(3)--node [below]{$\gamma$}(4);
\draw(4)--node [left]{$\delta$}(5);
\draw(5)--node [above]{$\epsilon$}(1);
\end{tikzpicture}  
$
\quad with the AR-quiver \quad 
$
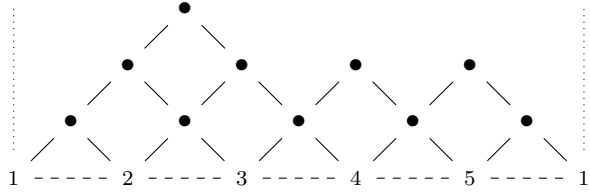
\begin{tikzpicture}[scale=.75]
\node(s1) at (0,0) {$\begin{smallmatrix} 1 \end{smallmatrix}$};
\node(s2) at (2,0) {$\begin{smallmatrix} 2 \end{smallmatrix}$}; 
\node(s3) at (4,0) {$\begin{smallmatrix} 3 \end{smallmatrix}$};
\node(s4) at (6,0) {$\begin{smallmatrix} 4 \end{smallmatrix}$};
\node(s5) at (8,0) {$\begin{smallmatrix} 5 \end{smallmatrix}$};
\node(s1') at (10,0) {$\begin{smallmatrix} 1 \end{smallmatrix}$};

\node(12) at (1,1) {$\bullet$};
\node(22) at (3,1) {$\bullet$};
\node(32) at (5,1){$\bullet$};
\node (42) at (7,1){$\bullet$};
\node (52) at (9,1){$\bullet$};

\node(13)at(2,2){$\bullet$};
\node(23)at (4,2){$\bullet$};
\node(33) at (6,2){$\bullet$};
\node (43)at (8,2){$\bullet$};

\node(14) at (3,3){$\bullet$};

\draw(s1)--(12)--(13)--(14);
\draw(s2)--(22)--(23);
\draw(s3)--(32)--(33);
\draw(s4)--(42)--(43);
\draw(s5)--(52);

\draw(12)--(s2);
\draw(13)--(22)--(s3);
\draw(14)--(23)--(32)--(s4);
\draw(33)--(42)--(s5);
\draw(43)--(52)--(s1');

\draw[dashed][-](s1)--(s2);
\draw[dashed][-](s2)--(s3);
\draw[dashed][-](s3)--(s4);
\draw[dashed][-](s4)--(s5);
\draw[dashed][-](s5)--(s1');

\draw[dotted][-](0,0.5)--(0,3);
\draw[dotted][-](10,0.5)--(10,3);
\end{tikzpicture}
$ \\

\noindent
The subcategory $\mathcal C_\Gamma$ contains a tilting module $T_\cC=P_1\oplus P_4\oplus P_5\oplus S_4\oplus \begin{smallmatrix} 4\\3 \end{smallmatrix}$. One can check that $\domdim \Gamma=2$, $\gldim\Gamma=4$, and $\gldim\End_\Gamma(T_\cC)=4$.
\end{exm}

In the next discussion, we are going to show exactly when it holds that $\gldim B_\cC<\gldim \Gamma$.
The proof relies on the following easy fact about homological dimensions. 
 
 \begin{lem}\label{max pd}
Let $\Lambda$ be an artin algebra with $\gldim \Lambda=d$. If $N$ is a $\Lambda$-submodule of $M$ with $\pd N=d$, then $\pd M=d$.
 \end{lem}

\begin{thm}\label{drop}
Let $\Gamma$ be an artin algebra with $\domdim \Gamma \geq 2$, 
$T_\cC$ be the unique tilting module in $\cC_\Gamma$, and $B_\cC=\End_\Gamma(T_\cC)^{op}$. Then $$\gldim B_\cC<\gldim\Gamma \text{ if and only if } \pd_\Gamma (\tau T_\cC) <\gldim\Gamma.$$
\end{thm}

\begin{proof}
Let $\gldim\Gamma=d$. Then: 

``$\Longleftarrow$:'' Suppose $\pd_\Gamma (\tau T_\cC) <d$. We prove by contradiction: assume $\gldim B_\cC=d$. So, there is a $B_\cC$-module $U$  with $\pd_{B_\cC} U=d$. We have an exact sequence of $B_\cC$-modules:
$$
0\rightarrow \Omega U \rightarrow \Hom_\Gamma(T_\cC,T_0) \rightarrow U\rightarrow 0.
$$
Then the first syzygy $\Omega U \cong \Hom_\Gamma(T_\cC,M)$, for some $M \in \Gen(T_\cC)$ by Lemma \ref{syzygy}. Thus, $\pd_{B_\cC} (\Hom_\Gamma(T_\cC,M)) =d-1$ and $\pd_\Gamma M=d$ due to Lemma $\ref{projdim key}$. 

As $\Hom_{B_\cC} (\Hom_\Gamma(T_\cC,M), \, \Hom_\Gamma(T_\cC,T_0)) \cong \Hom_\Gamma(M,T_0)$, the embedding $\Hom_\Gamma(T_\cC,M) \rightarrow \Hom_\Gamma(T_\cC,T_0)$ is induced by a morphism $f: M \rightarrow T_0$. Since $T_0$ is in $\Sub(\widetilde{Q})$ and $\pd_\Gamma M=d$, the map $f$ is not a monomorphism. But $\Hom_\Gamma(T_\cC,\Ker(f))=0$ since $\Hom_\Gamma(T_\cC,M) \rightarrow \Hom_\Gamma(T_\cC,T_0)$ is a monomorphism. Hence, $\Ker(f)\in\cF(T_\cC)=\Sub(\tau T_\cC)$. From Lemma $\ref{max pd}$ and our assumption $\pd_\Gamma (\tau T_\cC) <d$, we know $\pd_\Gamma \Ker(f)<d$. 

Consider the exact sequences:
$$
0\rightarrow \Ker(f)\rightarrow M \rightarrow \Ima(f)\rightarrow 0, \text{ and}
$$
$$
0\rightarrow \Ima(f)\rightarrow T_0 \rightarrow \cok(f)\rightarrow 0.
$$
Since $\pd_\Gamma \Ker(f)<d$ and $\pd_\Gamma M=d$, we have $\pd_\Gamma \Ima(f)=d$. However, by Lemma $\ref{max pd}$ and the above second exact sequence, it implies that $\pd_\Gamma T_0=d\geq2$, which is a contradiction. Therefore, we must have $\gldim B_\cC<d$. 

\vskip10pt
``$\Longrightarrow$:''  By Proposition \ref{gldim finite gor}, it follows that $\Gdim\Gamma=d$. Since $T_\cC$ is the direct sum of the first cosyzygy of the injective resolution of $\Gamma$ and $\widetilde Q$, it follows that $\id T_\cC=d-1$. According to \cite[Theorem 3.2]{GHPRU}, it follows that $\pd_\Gamma\tau T_\cC<d$.
\end{proof}
 
\begin{exm}
This example illustrates Theorem $\ref{drop}$. Let $Q$ be the quiver 
$$
\xymatrix{1&2\ar[l]&3\ar[l]&4\ar[l]&5\ar[l]}
$$
and $\Gamma=\bk Q/\rad^2(\bk Q)$. Then $\gldim\Gamma=4=\domdim\Gamma$, $\pd_\Gamma (\tau T_\cC)=0$ and $\gldim B_\cC=3$.
\end{exm}

%

Finally, we give a digression on the condition $\pd_\Gamma (\tau T_{\cC}) < \gldim\Gamma$.

\begin{lem} \label{pd<d}
Suppose $\Gamma$ is an artin algebra with $\gldim \Gamma=d$. Assume  the tilting module $T_\cC$ exists in $\cC_\Gamma$, then the following statements are equivalent:
\begin{enumerate}
\item $\pd_\Gamma (\tau T_\cC) <d$,
\item $\Ext_\Gamma^d(\tau T_\cC, M)=0$ for all $\Gamma$-modules $M$,
\item $\Ext_\Gamma^d(\tau T_\cC, S)=0$ for all simple $\Gamma$-modules $S$ such that $\id S=d$,
\item $\tau^{-1}(\Sigma^{d-1}S)\in \Gen(\widetilde Q)$, for all simple $\Gamma$-modules $S$ satisfying $\id S=d$, where $\widetilde Q$ is an additive generator of projective-injective $\Gamma$-modules and $\Sigma^{d-1}S$ is the $(d-1)$-th syzygy of $S$ in the minimal injective resolution.
\end{enumerate}
\end{lem}

\begin{proof}
$(1)\iff (2)$, $(2)\Longrightarrow (3)$ Trivial.

$(3)\Longrightarrow (2)$ We use induction on the length $l(M)$. If $l(M)=1$, $M$ is simple. If $\id M<d$, it is clear that $\Ext_\Gamma^d(\tau T_\cC, M)=0$. Otherwise $\id M=d$, the assertion follows directly. Now assume $l(M)>1$. There is a simple $\Gamma$-module $S$ and an exact sequence:
$$
0\rightarrow M'\rightarrow M\rightarrow S\rightarrow 0.
$$

Since $\l(M')<l(M)$ and $l(S)<l(M)$, then by induction hypothesis $\Ext_\Gamma^d(\tau T_\cC, M')=0$ and $\Ext_\Gamma^d(\tau T_\cC, S)=0$.  Hence $\Ext_\Gamma^d(\tau T_\cC, M)=0$. 

$(3)\iff (4)$ Notice that $\Ext_\Gamma^d(\tau T_\cC, S) \simeq \Ext_\Gamma^1(\tau T_\cC, \Sigma^{d-1}S)$. Because $\pd T_\cC=1$ and $\id\Sigma^{d-1}S=1$, it follows by \cite[IV,2.14]{ASS} that $\Ext_\Gamma^1(\tau T_\cC, \Sigma^{d-1}S) \simeq \Ext_\Gamma^1(T_\cC,\tau^{-1}(\Sigma^{d-1}S))$. Hence $\Ext_\Gamma^d(\tau T_\cC, S)=0$ if and only if $\Ext_\Gamma^1( T_\cC, \tau^{-1}(\Sigma^{d-1}S))=0$ if and only if  $\tau^{-1}(\Sigma^{d-1}S)\in\Gen(T_\cC)=\Gen(\widetilde Q)$.
\end{proof}

\begin{rmk}
Notice that if a simple $\Gamma$-module $S$ satisfies $\id S=\gldim \Gamma < \infty$, then its projective cover $P(S)$ is not injective.
\end{rmk}

Let
$
0\rightarrow S\rightarrow I_0(S)\rightarrow I_1(S)\rightarrow\cdots
$ be the minimal injective resolution of a simple module $S$, $\nu=D\Hom_\Gamma(-,\Gamma)$ be the \textbf{Nakayama functor} and $\nu^{-1}=\Hom_\Gamma(D\Gamma,-)$ be its quasi-inverse.

\begin{prop}\label{tau T equiv}
Suppose $\Gamma$ is an artin algebra with $\gldim\Gamma=d$. Assume  the tilting module $T_\cC$ exists in $\cC_\Gamma$, then $\pd_\Gamma (\tau T_\cC) <d$ if and only if $\nu^{-1}I_d(S)$ is injective, for any simple $\Gamma$-module $S$ with $\id_\Gamma S=d$.
\end{prop}

\begin{proof}
By Lemma \ref{pd<d}, $\pd (\tau T_\cC) <d$ if and only if $\tau^{-1}(\Sigma^{d-1}S)\in \Gen(\widetilde Q)$, for all simple modules $S$ such that $\id S=d$. 

Applying $\nu^{-1}$ to the following minimal injective resolution
$$
0\rightarrow \Sigma^{d-1}S\rightarrow I_{d-1}(S)\rightarrow I_d(S)\rightarrow 0, 
$$
we have an exact sequence:
$$
0\rightarrow \nu^{-1}\Sigma^{d-1}S\rightarrow \nu^{-1}I_{d-1}(S)\rightarrow \nu^{-1}I_d(S)\rightarrow \tau^{-1}(\Sigma^{d-1}S) \rightarrow 0.
$$
Hence the assertion follows from the fact that $\tau^{-1}(\Sigma^{d-1}S)\in \Gen(\widetilde Q)$ if and only if its projective cover $ \nu^{-1}I_d(S)$ is injective.
\end{proof}

\begin{prop}\label{tau T equiv 2}
Let $\Lambda$ be an artin algebra and $M \simeq \bigoplus_i M_i$ be a generator-cogenerator, where the $M_i$ are non-isomorphic indecomposable $\Lambda$-modules. Denote $\Gamma=\End_\Lambda(M)^{op}$. Then 
\begin{enumerate}
\item The complete set of representatives of non-isomorphic indecomposable projective $\Gamma$-modules is given by $\{D\Hom_\Lambda(M_i, M) \}.$
\item There are $\Gamma$-module isomorphisms $\nu^{-1}D\Hom_\Lambda(M_i,M)\simeq \Hom_\Lambda(M,M_i)$. 
\item The module $\nu^{-1}D\Hom_\Lambda(M_i,M)$ is injective if and only if $M_i$ is an injective $\Lambda$-module.
\end{enumerate}
\end{prop}

\subsection{A relation to the Finitistic Dimension Conjecture}
\label{findim conjecture}

In this section, we obtain an application of the results in Section \ref{endomorphism algebra} to the Finitistic Dimension Conjecture for a certain class of artin algebras of representation dimension at most 4. First, let us recall:

\begin{defn}
 Let $\Lambda$ be an artin algebra. Then the \textbf{finitistic dimension of $\Lambda$} is:
 $$
 \findim \Lambda := \sup \, \{\pd M \mid M\in \modd \Lambda \text{ and } \pd M<\infty\}.
 $$

\noindent
The \textbf{representation dimension of $\Lambda$} is:
 $$
 \repdim \Lambda := \inf \, \{\gldim \End_\Lambda(X) \mid X \text{\ is\ a\ generator-cogenerator in}\modd \Lambda \}.
 $$
 \end{defn}

Iyama \cite{I} proved that for any artin algebra $\Lambda$, its $\repdim\Lambda<\infty$ always. On the other hand, the long-standing Finitistic Dimension Conjecture says that for any artin algebra $\Lambda$, its $\findim \Lambda$ is finite. In \cite{IG}, Igusa and Todorov proved a partial result of this conjecture. In particular, they proved that $\findim\Lambda<\infty$ provided $\repdim\Lambda\leq3$. Their proof relied on the following result using the IT function $\psi$ defined as:

\begin{defn}\label{IT functions}
Let $\Lambda$ be an artin algebra. Let $K_0$ be the abelian group generated by $[X]$, for all finitely generated $\Lambda$-modules $X$, modulo the relations: \\
(a) $[C]=[A]+[B]$ if $C=A\oplus B$ and (b) $[P]=0$ for projective $\Lambda$-modules $[P]$.\\
Let $L:K_0\rightarrow K_0$
be 
the group homomorphism 
defined by $L[X]=[\Omega X]$.

For any $\Lambda$-module $M$, denote by $\langle \add M\rangle$ the subgroup of $K_0$ generated by $[M_i]$ where $M_i$'s are indecomposable summands of $M$. Then the \textbf{IT-functions} are:
$$
\phi(M):= \min\, \{m \mid L^m\langle\add M\rangle \cong L^{m+1}\langle\add M\rangle \}
$$
$$
\psi(M):=\phi(M)+\sup\, \{\pd X \mid \pd X<\infty, X \text{\ is\ a\ summand\ of\ } \Omega^{\phi(M)}M\}.
$$
\end{defn}

\begin{rmk}
For any $\Lambda$-module $M$, the IT-functions $\phi(M)$ and $\psi(M)$ are always finite. When $\pd_\Lambda M<\infty$, it is easy to see that $\phi(M)=\psi(M)=\pd_\Lambda M$. So IT-functions are generalizations of projective dimension. 
\end{rmk}

The next lemma is also a generalization of the well-known result about projective dimensions:
$\pd C\leq \pd(A\oplus B)+1$ 
when $\pd A$ and $\pd B$ are finite.

\begin{lem}\label{psi}
Suppose that $0\rightarrow A\rightarrow B\rightarrow C\rightarrow 0$ is a short exact sequence of finitely generated $\Lambda$-modules and $C$ has finite projective dimension. Then $\pd_\Lambda C\leq \psi(A\oplus B)+1$.
\end{lem}

\begin{cor}\cite{IG}\label{end P}
Let $\Gamma$ be an artin algebra with $\gldim\Gamma\leq3$. Let $\Lambda=\End_\Gamma(P)^{op}$, where $P$ is a projective $\Gamma$-module. Then $\findim\Lambda\leq \psi(\Hom_\Gamma(P,\Gamma))+3$, where $\Hom_\Gamma(P,\Gamma)$ is considered as a $\Lambda$-module. 
\end{cor}

Motivated by Igusa-Todorov's result, Wei introduced in \cite{W} the notion of \textbf{$m$-IT algebra}, for any non-negative integer $m$:

\begin{defn}\cite{W}
Let $\Lambda$ be an artin algebra. Then $\Lambda$ is said to be $m$-IT if there exists a module $V$ such that for any $\Lambda$-module $M$ there is an exact sequence
$$
0\rightarrow V_1\rightarrow V_0\rightarrow \Omega^m M\oplus P\rightarrow 0,
$$
where $V_0, V_1 \in \add V$ and $P$ is a projective $\Lambda$-module. 
\end{defn}

Applying Lemma~$\ref{psi}$, it is easy to see that the finitistic dimension of an $m$-IT algebra $\Lambda$ is bounded as $\findim \Lambda \leq m+1+\psi(V)$. Consequently, the Finitistic Dimension Conjecture holds for $m$-IT algebras, \cite[Theorem 1.1]{W}. 

It was shown that the class of $2$-IT algebras is closed under taking endomorphism algebras of projective modules, \cite[Theorem 1.2]{W}. 
Artin algebras with global dimension $d$ are $(d-1)$-IT.  If $\repdim\Lambda\leq3$, then $\Lambda$ is $2$-IT. However, there exist algebras which are not IT algebras, for example, the exterior algebra of a $3$-dimensional vector space.

\vskip10pt
In the following, we are going to see how the endomorphism algebra $B_\cC:=\End_{\Gamma}(T_\cC)^{op}$ studied in Section \ref{endomorphism algebra} relates to $2$-IT algebras.

To fix the notation, let $\Lambda$ be a basic artin algebra. Since $\repdim\Lambda<\infty$, there exists a (multiplicity-free) generator-cogenerator $X\in \modd \Lambda$ such that $\Gamma=\End_{\Lambda}(X)^{op}$ has finite global dimension, say $\gldim \Gamma=d$. Additionally, $\Gamma$ has dominant dimension at least $2$. Let $T_\cC$ be the unique tilting module in $\cC_\Gamma$. Denote by $\widetilde{Q}$ the sum of the representatives of isomorphism classes of indecomposable projective-injective $\Gamma$-modules. We have $\Lambda \cong \End_{\Gamma}(\widetilde{Q})^{op}$ by Proposition~\ref{morita Q}. 

%
%

From the endomorphism algebra $B_\cC:=\End_\Gamma(T_\cC)^{op}$, we can also recover the algebra $\Lambda$ as follows: Let $R:=\Hom_\Gamma(T_\cC,\widetilde{Q})$ which is a projective $B_\cC$-module. Then

\begin{lem}\label{BR}
$\End_{B_\cC} (R)^{op} \cong \Lambda$. 
\end{lem}

\begin{proof}
$\End_{B_\cC} (R)=\Hom_{B_\cC} (\Hom_\Gamma(T_\cC,\widetilde{Q}), \Hom_\Gamma(T_\cC,\widetilde{Q})) \cong \Hom_\Gamma(\widetilde{Q},\widetilde{Q})=\End_\Gamma(\widetilde{Q})$; see \cite[VI, \S3.2]{ASS}. Thus, $\Lambda \cong \End_{\Gamma}(\widetilde{Q})^{op} \cong \End_{B_\cC} (R)^{op}$. 
\end{proof}

Combining this fact with Corollary $\ref{end P}$ and \cite[Theorem 1.2]{W}, we have:

\begin{cor}\label{2 ig}
Suppose $\gldim B_\cC \leq 3$. Then $B_\cC$ is $2$-IT and hence $\Lambda \cong \End_{B_\cC} (R)^{op}$ is also $2$-IT, and $\findim \Lambda \leq \psi(\Hom_{B_\cC}(R,B_\cC))+3.$
\end{cor}


As a consequence of this result and Theorem \ref{drop}, we conclude with the following special case of the Finitistic Dimension Conjecture: for an artin algebra $\Lambda$, let $X$ be its Auslander generator-cogenerator. Let $\Gamma := \End_\Lambda(X)^{op}$, then $\Gamma$ is of dominant dimension at least $2$ and satisfies $\gldim \Gamma = \repdim \Lambda$. Let $T_\cC$ be the unique tilting module in $\cC_\Gamma$. Then:
\begin{cor}\label{findim}
If $\repdim \Lambda\leq 4$ and $\pd_\Gamma (\tau T_\cC) \leq3$, then $\findim \Lambda<\infty$. 
\end{cor}

\begin{proof}
With our setting and assumptions, $ \gldim \Gamma = \repdim \Lambda \leq 4$ and $\pd_\Gamma (\tau T_\cC) \leq 3$. Let $B_\cC=\End_\Gamma(T_\cC)^{op}$. By Theorem $\ref{drop}$, we must have $\gldim B_\cC \leq 3$. It follows from Corollary $\ref{2 ig}$ that $\findim \Lambda \leq \psi(\Hom_{B_\cC}(R,B_\cC))+3$, which is finite. 
\end{proof}

\begin{exm}
\label{beilinson}
We give an example of an algebra of representation dimension $4$ which satisfies the hypothesis in Corollary $\ref{findim}$ and hence the set of such algebras is not empty. 
Let $\Lambda$ be the Beilinson algebra with $3$ vertices, that is, it is defined by the following quiver
$$
\xymatrix{3\ar[r]^{x_2}\ar@/_1pc/[r]_{x_1}\ar@/^1.2pc/[r]^{x_3}&2\ar[r]^{x_2}\ar@/_1pc/[r]_{x_1}\ar@/^1.2pc/[r]^{x_3}&1}
$$
with relations $I=\langle x_ix_j - x_jx_i \rangle$ for all $1\leq i,j\leq 3$. It was studied and shown by Krause-Kussin, Iyama, and Oppermann that $\repdim\Lambda=4$ (see for example \cite[Examples 7.3 and A.8]{O}). To check that $\Lambda$ satisfies the hypothesis in Corollary $\ref{findim}$, one needs to apply and check Propositions \ref{tau T equiv} and \ref{tau T equiv 2}(3). 
\end{exm}

\section{Special class: Extensions of Auslander algebras by injective modules}\label{secA[I]}\label{Auslander}\label{A[I]}

We now describe a procedure to create algebras $\Lambda$ which will have tilting modules that are generated and cogenerated by projective-injective modules, that is, those tilting modules are in $\mathcal C_{\Lambda}$. 
In particular, we describe a class of algebras, constructed from Auslander algebras by ``extending" the Auslander algebras by certain injective modules. 

\subsection{General triangular matrix construction} 

We now recall the construction of an algebra $\Lambda$ from algebras $R$ and $S$ and a bimodule $_SM_R$ as investigated in \cite{FGR}. We also recall some basic properties of these algebras, together with the description of modules over such algebras. 

\begin{defn} \label{triangular} Let $R$ and $S$ be finite dimensional algebras over the field $k$. Let $_SM_R$  be an $S$-$R$-bimodule. Define an algebra $\Lambda$, which we will often denote by $\Lambda=T(R,S, _S\!M_R)$: 
$$\Lambda=T(R,S, _S\!M_R):=\left[\begin{matrix}R&0\\_SM_R&S\end{matrix}\right],$$
where the multiplication is defined using the bimodule structure of $_SM_R$.
\end{defn}

A convenient way of viewing $\Lambda$-modules is using the category of triples $\mathcal T$ in \cite{FGR}. We now recall this definition and some of the basic properties.

\begin{defn} \label{triples} Let $\Lambda=T(R,S,  _S\!M_R)$. Define the category of triples $\mathcal T$ as follows:\\
Objects of $\mathcal T$ are triples: 
$(_RW,\ _S\!V,\  f\!:\ _S\!M_R\otimes_RW \to\ _S\!V)$, where $_RW$ and $_SV$ are left $R$ and $S$-modules respectively, and $f$ is an $S$-homomorphism.
Morphisms in $\mathcal T$ are pairs:
 \\ $(\alpha, \beta): (_RW,\ _S\!V, \ f\!:\ _S\!M_R \otimes_RW\to \ _S\!V) \to (_RW', \ _S\!V', \ f'\!: \ _S\!M_R\otimes_RW' \to \ _S\!V')$, where $\alpha: \ _R\!W\to \ _R\!W'$ is an $R$-homomorphism and $\beta: \ _S\!V\to \ _S\!V'$ is an $S$-homomorphism making appropriate diagrams commute.
\end{defn}

\begin{rem} Let $\Lambda=T(R,S, _S\!M_R)$ and let $\mathcal T$ be the associated category of triples. Then the categories $\modd \Lambda$ and $ \mathcal T$ are equivalent. Using this fact we refer to triples as $\Lambda$-modules.
\end{rem}

\begin{prop}\label{Reiten}\cite[Proposition III.2.5]{ARS}, \cite{FGR} Let $\Lambda=T(R,S, _S\!M_R)$ and let $\mathcal T$ be the associated category of triples.
\begin{enumerate}
\item Indecomposable projective objects in $\mathcal T$ are $(0,_S\!P,0)$ and $(_RQ,\   _S\!M_R\otimes _R\!Q,\  Id_{_S\!M_R\otimes _R\!Q} )$, where $_SP$ and $_RQ$ are indecomposable projective $S$-modules and $R$-modules respectively.
\item Indecomposable injective objects in $\mathcal T$ are $(_RJ,0,0)$ and \\
$(\Hom_S(M,I),\ _SI,\ \eta:\ _SM_R\otimes\Hom_S(M,I)\ \stackrel{\simeq}\longrightarrow\ _SI)$,
where $_RJ$ and $_SI$ are indecomposable injective $R$-modules and $S$-modules respectively and $\eta(m\otimes f):=f(m)$ is an $S$-isomorphism.
\end{enumerate}
\end{prop}

The above proposition has a description of all indecomposable projective and injective objects in $\mathcal T$ and hence, using equivalence, projective and injective $\Lambda$-modules. In addition to this, we will also be using the following functor which relates categories of $S$ and $\Lambda$-modules.

\begin{prop}\label{Psi} Let $\Lambda=T(R,S, _S\!M_R)$. Then $\Psi(_SX):=(0,_S\!X, 0)$ defines a functor
$\Psi: \modd S \to \modd \Lambda,$
which has the following properties: 
\begin{enumerate}
\item $\Psi$ is fully-faithful.
\item $\Psi$ preserves kernels.
\item $\Psi$ preserves projective resolutions.
\end{enumerate}
\end{prop}

\subsection{Triangular matrix construction from Auslander algebras} 

In this section, we will look at the triangular matrix where $S=A$ is an Auslander algebra, $_AE$ is a special injective $A$-module and $R=\End_A(E)^{op}$, then $_AE_R$ is an $A$-$R$-bimodule. For the simplicity of notation we will denote the algebra $T(R,A, _A\!E_R)$ by $A[E]$.


\begin{defn} \label{A[I]}
 Let $A$ be an Auslander algebra and let $\widetilde{Q}=\bigoplus _{i=1}^t  Q_i$, where
$\{Q_1,\dots, Q_t\}$ is a set of representatives of isomorphism classes of all indecomposable projective-injective $A$-modules. 
 We choose an $A$-module $E$ which satisfies the following conditions:
\begin{enumerate}
\item $E=I_1\oplus \dots \oplus I_r$, where the $I_i$ are indecomposable injective $A$-modules for all $i$,
\item  $\End_{A}(I_i)=K_i$, where $K_i$ is a field,
\item  $\Hom_{A}(I_i, I_j)=0$ for all $i\neq j$, 
\item  $\Hom_A(E,\widetilde{Q})=0$.
\end{enumerate}
Let $A[E]: = T(R, A, _A\!E_R)$ be the triangular matrix algebra as in Definition \ref{triangular}. 
$$A[E]:=\left[\begin{matrix}R&0\\ _AE_R&A\end{matrix}\right].$$
\end{defn}

We now use the fact that the category of $A[E]$-modules is equivalent to the category of triples $(_RW,\ _A\!V,\  f\!:\ _A\!E_R\otimes_RW\to\ _A\!V)$, as described in Definition \ref{triples}. 

\begin{lem} Let $E$, $I_i$, $K_i$ and $A[E]$ be as above.
 The representatives of the isomorphism classes of indecomposable projective-injective $A[E]$-modules correspond to the following triples:
\begin{enumerate}
\item There are $t$ indecomposable projective-injective modules of type $(0, _A\!Q_i, 0)$, where $Q_i$ are the indecomposable projective-injective $A$-modules, and
\item  There are $r$ indecomposable projective-injective modules of type \\ $(K_i, \ _A\!I_i, \ \eta_i\!:\ _A\!E_R\otimes_RK_i \xrightarrow{\cong} \ _A\!I_i)$, where $I_i$ are the indecomposable summand of the $A$-module $E$.

\end{enumerate}
\end{lem}
\begin{proof} (1) Using Proposition \ref{Reiten}(1), it is clear that the $A[E]$-modules $(0, _A\!Q_i,0)$ are indecomposable and projective. To see that they are injective: by Proposition \ref{Reiten}(2), the indecomposable injective $A[E]$-modules are given as
$$(\Hom_A(E,Q_i),\ _AQ_i,\ \eta_i:\ _AE_R\otimes\Hom_A(E,Q_i)\ \to\ _AQ_i),$$ 
which are equal to $(0,_A\!Q_i,0)$, using the fact that $\Hom_A(E,\widetilde{Q})=0$ by condition (4) in Definition \ref{A[I]}. Therefore, $(0,_A\!Q_i,0)$ are indecomposable projective-injective $A[E]$-modules. 

(2) By Proposition \ref{Reiten}(2) and conditions (2) and (3) in Definition \ref{A[I]}, it is clear that the $A[E]$-modules $(K_i, \ _A\!I_i, \ \eta_i\!:\ _A\!E_R\otimes_RK_i \xrightarrow{\cong} \ _A\!I_i)$ are indecomposable and injective. They are also projective by Proposition \ref{Reiten}(1) and the fact that $K_i$ are projective $R$-modules.
\end{proof}

By \cite[Lemma 1.1]{CBS} and Corollary \ref{cotilting cor}, we know that for an Auslander algebra $A$, the tilting module $T_\cC$ exists in $\cC_{A}$. We now relate this tilting module $T_\cC$ to a module in $\cC_{A[E]}$ and show that there exists a tilting module in $\cC_{A[E]}$.

\begin{lem} Let $T_\cC$ be a tilting module in $\cC_{A}$. Then $(0, T_\cC, 0)$ is a module in  
$\cC_{A[E]}$.
\end{lem}
\begin{proof} 
The $A[E]$-module $(0, T_\cC, 0)$ is a submodule and a quotient module of the projective-injective $A[E]$-modules since  $T_\cC$ is submodule and quotient module of the modules in $\add \widetilde{Q}$.
\end{proof}

\begin{lem}  Let $T_\cC$ be a tilting module in $\cC_{A}$. Then 
\begin{enumerate}
\item $\pd _{A[E]}(0, T_\cC, 0)\leq 1$,
\item $\Ext^1_{A[E]}((0,T_\cC,0),(0,T_\cC,0))=0$.
\end{enumerate}
\end{lem}
\begin{proof}  Part (1) follows from Proposition $\ref{Psi}$ (3). Part (2): 
  Proposition $\ref{Psi}$ (1) and (3) implies that  $\Ext^1_{A[E]}((0,T_\cC,0),(0,T_\cC,0))=\Ext_A^1(T_\cC,T_\cC)=0$.
\end{proof}

\begin{cor} Let $T_\cC$ be a tilting module in $\cC_{A}$. Then $(0, T_\cC, 0)$ is a partial tilting module in  
$\cC_{A[E]}$, with $n_A$ summands, where $n_A$ is the number of non-isomorphic simple $A$-modules.
\end{cor}
 
 
\begin{thm} Let $A$ be an Auslander algebra and $A[E]$ be the algebra described in Definition $\ref{A[I]}$. Then there is a tilting module in $\cc_{A[E]}$.
\end{thm}
\begin{proof} Let $T_\cC$ be a tilting module in $\cc_A$. Then $T_{\cC_{A[E]}}:=(0,T_\cC,0)\oplus \left(\bigoplus_{i=1}^rY_i\right)$ is a tilting module in $\cc_{A[E]}$, where $Y_i = (K_i,\ _AI_i,\ \eta_i:\! \ _AE_R\otimes_RK_i \xrightarrow{\cong} \ _AI_i)$. The number of indecomposable summands of $T_{\cC_{A[E]}}$ equals $n_A+r$ which is the number of non-isomorphic simple $A[E]$-modules.
\end{proof}

\section{Special class: Tilting modules in $\mathcal C_{\Lambda}$ for Nakayama algebras} \label{Nakayama} 

\subsection{Nakayama algebras} 
\label{subsec:Nakayama}

In this section, let $\Lambda$ be any Nakayama algebra. We will show criteria for the subcategory $\mathcal C_{\Lambda}$ to contain a tilting module $T_\cC$. Due to Theorem $\ref{tilting}$, it is equivalent to finding Nakayama algebras with dominant dimension at least $2$. Notice that such a class of algebras has been classified by Fuller in \cite[Lemma 4.3]{F} in a module theoretic way. However, using Auslander-Reiten theory, our descriptions in Corollary~\ref{domdim cn} and Theorem~\ref{domdim ln} can be regarded as a combinatorial approach.

First, we recall some well known facts about Nakayama algebras.
A module $M$ over an artin algebra is called a \textbf {uniserial module} if the set of its submodules is totally ordered by inclusion, or equivalently, there is a unique composition series of $M$. An artin algebra $\Lambda$ is said to be \textbf{Nakayama algebra} if both the indecomposable projective and indecomposable injective modules are uniserial. One can show that all indecomposable modules over a Nakayama algebra are uniserial \cite[VI, Theorem 2.1]{ARS}. 

Moreover, we have the following classification of Nakayama algebras \cite[Theorem V.3.2]{ASS}: A basic connected artin algebra $\Lambda$ is a Nakayama algebra if and only if its ordinary quiver $Q_\Lambda$ is either a quiver of type $A_n$ with straight orientation or a complete oriented cycle.   According to \cite{Ma}, Nakayama algebras whose ordinary quiver is $A_n$ with straight orientation are called \textbf{Linear-Nakayama algebras} and Nakayama algebras whose ordinary quiver is a complete oriented cycle are called the \textbf{Cyclic-Nakayama algebras}. In this section, we always assume $\Lambda$ to be basic and connected.

For any $\Lambda$-module $M$, denote by $l(M)$  the length of $M$. For a Nakayama algebra $\Lambda$, there exists an ordering $\{P_1, P_2, \ldots, P_n\}$ of non-isomorphic indecomposable projective $\Lambda$-modules such that:
\begin{itemize}
 \item[(a)] $P_{i+1}/\rad P_{i+1}\cong \tau^{-1}(P_i/\rad P_i)$, for $1\leq i\leq n-1$; and \\
 if $l(P_1)\neq 1$, then $P_1/\rad P_1\cong (P_n/\rad P_n)$, \vskip 7pt

 \item[(b)] $l(P_i)\geq2$, for $2\leq i\leq n$, \vskip 7pt

 \item[(c)] $l(P_{i+1})\leq l(P_i)+1$, for $1\leq i\leq n-1$ and $l(P_1)\leq l(P_n)+1$.
\end{itemize}

Such an ordering is called a \textbf{Kupisch series} for $\Lambda$ and $(l(P_1),l(P_2), \ldots, l(P_n))$ is called the corresponding \textbf{admissible sequence} for $\Lambda$. 

\begin{rmk}\label{Kupisch}\ 
\begin{enumerate}
\item The Kupisch series (and hence the admissible sequence) for a Nakayama algebra is always unique up to a cyclic permutation (or simply unique if $l(P_1)=1$).
\item Let $(c_1,c_2,\ldots, c_n)$ be a sequence of integers such that $c_j\geq 2$ for all $j\geq 2$, and $c_{j+1}\leq 1+c_j$ for $j\leq n-1$, and $c_1\leq c_n+1$. There is a Nakayama algebra $\Lambda$ such that $(c_1,c_2,\ldots, c_n)$ is the admissible sequence for $\Lambda$.  
\end{enumerate}
\end{rmk}

\subsection{Cyclic-Nakayama algebras with dominant dimension at least $2$} 

Suppose $\Lambda$ is a Cyclic-Nakayama algebra with $n$ simple modules. We always label the vertices of its ordinary quiver in such a way that arrows are $(i+1\rightarrow i)$ for $1\leq i\leq n-1$ and $(1\rightarrow n)$. It is easy to check that  $\{P_1,P_2,\ldots, P_n\}$ is a Kupisch series, where $P_i:=P(S_i)$ is the projective cover of the simple module $S_i$. 

Let $\Lambda$ be a Cyclic-Nakayama algebra with Kupisch series $(P_1,P_2,\ldots, P_n)$. Then we can view the corresponding admissible sequence $(c_1, c_2,\ldots ,c_n)$ as a function $c:\mathbb Z_n\rightarrow \mathbb Z$ sending $i\mapsto c_i$ and satisfying $c_{i+1}\leq c_i+1$ and $c_i\geq2$. On the other hand, each such function gives rise to a Cyclic-Nakayama algebra.

According to our labeling, it is easy to see:
\begin{lem} \label{soc P}
Let $P_i:=P(S_i)$ be the projective cover of the simple module $S_i$. Then 
\begin{enumerate}
\item $S_i\cong\tau S_{i+1}$,
\item $\soc P_i\cong S_{i-c_i+1}$, where the index $i-c_i+1$ is regarded as an element in $\mathbb Z_n$.
 \end{enumerate}
\end{lem}

\begin{lem}\label{quadrilateral}
Suppose $P_i:=P(S_i)$ is projective non-injective. The injective envelope $I(P_i)$ is a projective-injective module. Then $\soc I(P_i)/P_i\cong S_{i+1}$, where the index $i+1$ is regarded as an element in $\mathbb Z_n$.
 \end{lem}
 \begin{proof}
 The exact sequence:
 $
0\rightarrow P_i\rightarrow I(P_i)\rightarrow I(P_i)/P_i\rightarrow 0
$
suggests that $\soc I(P_i)/P_i\cong \tau^{-1}\Top P_i=\tau^{-1}S_i.$
Then the assertion follows from Lemma $\ref{soc P}$.
\end{proof}

\begin{defn}\label{Nakayama numerical}
Define $\cQ_c:=\{ i \in\mathbb Z_n \mid c_{i+1}\leq c_i \}$ and $\cP_c:=\{i\in\mathbb Z_n \mid c_{i+1}=c_i+1\}$. 
\end{defn}

By definition, $\cQ_c\cup\cP_c=\mathbb Z_n$. An indecomposable projective module $P_i$ is also injective if and only if $i\in \cQ_c$.

\begin{thm}
Let $P_i$ be an indecomposable projective module with the index $i\in \cP_c$. Then \\ $\domdim P_i\geq2$ if and only if $i\in\{j-c_{j}\in\mathbb Z_n \mid j\in\cQ_c\}$.
\end{thm}
\begin{proof}
The dominant dimension $\domdim P_i\geq2$ if and only if $I(P_i)/P_i$ is a submodule of $P_j$ for some $j\in\cQ_c$, which is equivalent to 
$$
\soc I(P_i)/P_i\cong \soc P_j.
$$
By Lemmas $\ref{soc P}$ and $\ref{quadrilateral}$, it is equivalent to say $i+1=j-c_{j}+1$. Therefore, $\domdim P_i\geq2$ if and only if $i\in\{j-c_{j}\in\mathbb Z_n \mid j\in\cQ_c\}$.
\end{proof}
 
\begin{cor}\label{domdim cn}
Suppose $\Lambda$ is a Cyclic-Nakayama algebra with $n$ simple modules. Then $\domdim\Lambda\geq2$ if and only if $\cP_c\subseteq\{j-c_{j}\in\mathbb Z_n \mid j\in\cQ_c\}$.
\end{cor}

\begin{cor}
If $\domdim\Lambda\geq 2$ then $|\cQ_c|\geq \frac{n}{2}$.
\end{cor}

%

At last, we point out that this provides us with a method to find all Cyclic-Nakayama algebras with dominant dimension at least $2$.

Suppose $c$ and $c'$ are admissible sequences of Cyclic-Nakayama algebras $\Lambda$ and $\Lambda'$. We say that $c$ and $c'$ are in the same \textbf{difference class} (see \cite{Ma}) if  $c'_i=c_i+n$ for all $i$. From Corollary $\ref{domdim cn}$, it is easy to see that if $c$ and $c'$ are in the same difference class, then $\domdim\Lambda\geq 2$ if and only if $\domdim\Lambda'\geq 2$. In fact, the dominant dimension of $\Lambda$ only depends on the difference class of the admissible sequence \cite[Theorem 1.1.4]{Ma}. 

Therefore, to find all the Cyclic-Nakayama algebras with dominant dimension at least $2$, it is enough to find those with ``minimal'' admissible sequences. 

\begin{defn}
Suppose $c$ is an admissible sequence of Cyclic-Nakayama algebras $\Lambda$. We say that $c$ is \textbf{elementary} if $\mathop{\min}\limits_{1\leq i\leq n} \{c_i\} \leq n+1$, and $c$ is \textbf{absolutely elementary} if $\mathop{\min}\limits_{1\leq i\leq n} \{c_i\} = 2$.
\end{defn}

%
\begin{exm}
When $n=3$, the following are all the (absolutely) elementary admissible sequences for Cyclic-Nakayama algebras with dominant dimension at least $2$ (up to cyclic permutations): 

Absolutely elementary: $(2,2,2)$, $(2,2,3)$. 

Elementary: $(2,2,2)$, $(3,3,3)$, $(4,4,4)$, $(2,2,3)$, $(3,3,4)$. 
\end{exm}

\subsection{Linear-Nakayama algebras with dominant dimension at least $2$} 

Most of the results for Cyclic-Nakayama algebras also works for Linear-Nakayama algebras. For completeness, we will state the criteria for Linear-Nakayama algebras $\Lambda$ having $\domdim \Lambda \geq 2$.

Suppose $\Lambda$ is a Nakayama algebra whose underlying quiver is of type $A_n$ with $n$ simple modules. We always label the vertices of its quiver in such a way that arrows are $(i+1\rightarrow i)$ for $1\leq i\leq n-1$. It is easy to check that $\{P_1,P_2,\ldots, P_n\}$ is a Kupisch series, where $P_i:=P(S_i)$ is the projective cover of simple module $S_i$. 

The corresponding admissible sequence $(c_1,c_2,\ldots, c_n)$ satisfies  $c_1=1$, $c_{i+1}\leq c_i+1$, and $c_i\geq2$ for $2\leq i\leq n$. On the other hand, each such sequence gives rise to a Linear-Nakayama algebra. Define $\cQ_c:=\{ i \mid c_{i+1}\leq c_i \}$ and $\cP_c:=\{i \mid c_{i+1}=c_i+1\}$ as in Definition \ref{Nakayama numerical}. Notice that for Linear-Nakayama algebras $c_2=2$, $c_i\leq i$ and $1\in \cP_c$. 

\begin{thm}\label{domdim ln}
Suppose $\Lambda$ is a Linear-Nakayama algebra with $n$ simple modules. Then $\domdim\Lambda\geq2$ if and only if $\cP_c\subseteq\{j-c_j \mid j \in \cQ_c\}$.
\end{thm}

\subsection{The tilting module $T_\cC$ for Nakayama algebras}\label{tilting Nakayama}
Lastly, for a Nakayama algebra $\Lambda$, we give a description of the tilting module $T_\cC$ in the subcategory $\cC_\Lambda$ if it exists. 

Let $\Lambda$ be a Nakayama algebra with Kupisch series $(P_1,P_2,\ldots, P_n)$ and admissible sequence $(c_1, c_2,\ldots ,c_n)$. Let $\cQ_c$ and $\cP_c$ be the sets as defined in the previous sections. Then for each $i\in\cP_c$, define $\delta(i):= \min \, \{k\in \mathbb N \mid {i+k}\in\cQ_c \}$.

According to Corollary \ref{T_C compute}, we have the following description of the tilting module $T_\cC$:
\begin{thm}
Let $\Lambda$ be a Nakayama algebra with Kupisch series $(P_1,P_2,\ldots, P_n)$ and admissible sequence $(c_1, c_2,\ldots ,c_n)$. If the tilting module $T_\cC$ exists in $\cC_\Lambda$, then  
$$T_\cC \simeq \left(\mathop{\bigoplus}\limits_{j \in\cQ_c}{P_j}\right) \oplus \left(\mathop{\bigoplus}\limits_{i \in\cP_c} T_i \right),$$
where each $T_i$ is a uniserial module with socle $\soc T_i= i+1$ and length $|T_i|=\delta(i)$.
\end{thm}

\end{document}